\documentclass[10pt, twoside, open=right, pointlessnumbers, a4paper]{scrartcl}

\usepackage{amsmath}
\usepackage{titlesec}
\usepackage{wrapfig}
\usepackage{graphicx}
\usepackage{color}
\usepackage{xcolor, framed}
\usepackage{soul}
\usepackage[latin1]{inputenc}

\usepackage{multicol}
\usepackage{needspace}
\let\emph\undefined
\newcommand{\emph}[1]{\textsl{#1}}
\usepackage{tocloft}
\usepackage{amsmath}
\usepackage{amssymb}
\usepackage{amsfonts}
\usepackage{upgreek}
\usepackage{amsthm}
\usepackage{mathtools}
\usepackage{mathrsfs}
\usepackage{pifont}
\DeclareMathAlphabet{\mathpzc}{OT1}{pzc}{m}{it}
\usepackage{bbm}
\usepackage{bbding}
\usepackage{setspace}
\usepackage{mathdots}
\usepackage{relsize}
\usepackage{tocloft}

\numberwithin{equation}{section}
\usepackage{mathtools}
\mathtoolsset{showonlyrefs}
\numberwithin{equation}{section}
\usepackage{enumitem} 

\newenvironment{myenumerate}{\begin{enumerate}[topsep=2pt,parsep=2pt,partopsep=2pt,itemsep=0pt,label={\normalfont(\alph*)}]\itemsep0pt}{\end{enumerate}}
\newenvironment{xenumerate}{\begin{enumerate}[topsep=2pt,parsep=2pt,partopsep=2pt,itemsep=0pt,label={\normalfont(\arabic*)}]\itemsep0pt}{\end{enumerate}}
\newenvironment{pnum}{\begin{enumerate}[topsep=2pt,parsep=2pt,partopsep=2pt,itemsep=0pt,label={(\roman{*})}]}{\end{enumerate}}

\newtheoremstyle{style1}% name of the style to be used
  {13pt}% measure of space to leave above the theorem. E.g.: 3pt
  {13pt}% measure of space to leave below the theorem. E.g.: 3pt
  {}% name of font to use in the body of the theorem
  {}% measure of space to indent
  {\normalfont\bfseries}% name of head font
  {.}% punctuation between head and body
  {.5em}% space after theorem head; " " = normal interword space
  {}
\theoremstyle{style1}

\newtheorem{definition}[equation]{Definition}
\newtheorem{example}[equation]{Example}
\newtheorem{remark}[equation]{Remark}
\newtheorem{preremarks}[equation]{Remarks}

\newenvironment{remarks}[1][]{\begin{preremarks} \begin{myenumerate} }{  \end{myenumerate} \end{preremarks} }

\newtheoremstyle{style2}% name of the style to be used
  {13pt}% measure of space to leave above the theorem. E.g.: 3pt
  {13pt}% measure of space to leave below the theorem. E.g.: 3pt
  {\slshape}% name of font to use in the body of the theorem
  {}% measure of space to indent
  {\normalfont\bfseries}% name of head font
  {.}% punctuation between head and body
  {.5em}% space after theorem head; " " = normal interword space
  {}

\theoremstyle{style2}

\newtheorem{lemma}[equation]{Lemma}
\newtheorem{theorem}[equation]{Theorem}
\newtheorem{proposition}[equation]{Proposition}
\newtheorem{corollary}[equation]{Corollary}

\makeatletter
\deffootnote{1em}{1em}{%
\textsuperscript{\thefootnotemark\ }}
\usepackage{scrpage2}
\addtokomafont{sectioning}{\rmfamily}
\setcapindent{2em}
\usepackage{hhline}
\usepackage{colortbl}
\usepackage{tabularx}
\usepackage{pgf}
\usepackage{tikz}
\usepackage{fancybox}
\usetikzlibrary{arrows,shapes,trees,matrix}
\usetikzlibrary{decorations.markings}
\usetikzlibrary{knots}
\usepackage{pictexwd}
\usepackage{hyperref}

\newcommand{\myforall}{\quad \text{for all}\quad }

\newcommand{\Sym}{\operatorname{Sym}}
\newcommand{\HSym}{\operatorname{HSym}}
\newcommand{\id}{\operatorname{id}}

\newcommand{\catf}[1]{{\normalfont\textbf{#1}}}

\newcommand{\SpanGrpd}{\catf{SpanGrpd}}
\newcommand{\FinGrpd}{\catf{FinGrpd}}

\newcommand{\Vect}{\catf{Vect}}
\newcommand{\FinVect}{\catf{FinVect}}
\newcommand{\RepGrpd}[1]{\catf{VecBun}_{#1}\catf{Grpd}}

\newcommand{\Cob}{\catf{Cob}}

\newcommand{\cFrob}{\catf{cFrob}}

\newcommand{\spr}[1]{\left\langle #1 \right\rangle}
\newcommand{\tr}{\operatorname{tr}}
\newcommand{\PBun}{\catf{PBun}}
\newcommand{\Par}{\operatorname{Par}}

\newcommand{\sphere}{\mathbb{S}}
\newcommand{\Aut}{\operatorname{Aut}}
\newcommand{\End}{\operatorname{End}}
\newcommand{\td}{\text{\normalfont d}}

\let\Phi\undefined\DeclareMathSymbol{\Phi}{\mathalpha}{letters}{"08}
\let\Psi\undefined\DeclareMathSymbol{\Psi}{\mathalpha}{letters}{"09}
\let\Sigma\undefined\DeclareMathSymbol{\Sigma}{\mathalpha}{letters}{"06}
\let\Xi\undefined\DeclareMathSymbol{\Xi}{\mathalpha}{letters}{"04}
\let\Pi\undefined\DeclareMathSymbol{\Pi}{\mathalpha}{letters}{"05}
\let\Gamma\undefined\DeclareMathSymbol{\Gamma}{\mathalpha}{letters}{"00}
\let\Omega\undefined\DeclareMathSymbol{\Omega}{\mathalpha}{letters}{"0A}
\let\Lambda\undefined\DeclareMathSymbol{\Lambda}{\mathalpha}{letters}{"03}

\newcommand{\spaceplease}{\needspace{5\baselineskip}}

\newcommand{\extremespace}{\needspace{14\baselineskip}}

\let\to=\longrightarrow
\let\mapsto=\longmapsto
\newcommand{\VecBun}{\catf{VecBun}}

\newcommand{\Grpd}{\catf{Grpd}}
\newcommand{\sSet}{\catf{sSet}}
\newcommand{\colim}{\operatorname{colim}}

\usepackage{geometry}

\usepackage{epstopdf}

\automark{section}
\lohead{Orbifold Construction for Equivariant Topological Field Theories}\rehead{\leftmark}\ohead{\pagemark}
\cftsetindents{section}{1.5em}{2.8em}
\cftsetindents{subsection}{4.3em}{3.2em}
\ofoot{}

\begin{document}\thispagestyle{empty}
\setlength{\parskip}{0pt}
\setlength{\abovedisplayskip}{7pt}
\setlength{\belowdisplayskip}{7pt}
\setlength{\abovedisplayshortskip}{0pt}
\setlength{\belowdisplayshortskip}{0pt}

\begin{flushright}
\textsf{ZMP-HH/17-14}\\
\textsf{Hamburger Beiträge zur Mathematik Nr. 656}\\
\textsf{May 2017}
\end{flushright}

\vspace*{1cm}

\begin{center}
\Large \textbf{Orbifold Construction for Topological Field Theories} \normalsize \\[4ex] Christoph Schweigert and Lukas Woike
\end{center}

\begin{center}
\emph{Fachbereich Mathematik,   Universität Hamburg}\\
\emph{Bereich Algebra und Zahlentheorie}\\
\emph{Bundesstra{\ss}e 55,   D -- 20 146  Hamburg}
\end{center}

\begin{abstract}
\noindent \textbf{Abstract.} 
An equivariant topological field theory is defined on a cobordism category
of mani\-folds with principal fiber bundles for a fixed (finite) structure group.
We provide a geometric construction which for any given morphism $G \to H$ of finite groups assigns in a functorial way to
a $G$-equivariant topological field theory an $H$-equivariant topological field theory, the pushforward theory.
When $H$ is the trivial group, this yields an orbifold construction for $G$-equivariant topological field theories
which unifies and generalizes several known algebraic notions of
orbifoldization. 
\end{abstract}

\vspace*{1cm}
\tableofcontents

\spaceplease\section{Introduction and summary}
In several contexts, the following strategy has proven to be successful to construct
new mathematical objects from a given one: 
For a fixed (and for our purposes always finite) group $G$ find an embedding of the given object into a $G$-equivariant one (in the terminology of physics, this is is often referred to as the \emph{addition of twisted sectors}). 
 Then taking invariants in the appropriate sense yields a new object, the \emph{orbifold object}. In \cite[p.~495]{dvvvorbifold}, this has been phrased as follows:
``First the idea of an orbifold clearly implies that we keep only
the $G$-invariant states in the original Hilbert space $\mathscr{H}_0$.
However, [...] we also have to include twisted sectors.''

Examples include the construction of Frobenius algebras from $G$-equivariant
Frobenius algebras, see \cite{kaufmannorb}, and the construction of modular tensor categories from
$G$-equivariant modular tensor categories, see \cite{kirrolovg04}. 
A prominent special case of the latter is the orbifoldization of the trivial modular
tensor category, i.e.\ of the category of finite-dimensional vector spaces. One way of adding twisted sectors
yields a category of finite-dimensional $G$-graded vector spaces; taking invariants gives the representation category of the Drinfeld
double of $G$, see \cite[Example~5.2]{kirrolovg04}, which is an algebraic object of independent interest.

The study of these algebraic structures can be rather involved. However,
they are often accessible from a different point of view in the sense that they define a 
topological field theory in a certain dimension. 
Since topological field theories are sometimes easier to manipulate than the mere algebraic objects, this perspective can lead to to conceptual insights, unifications and generalizations of existing algebraic notions or simplifications of proofs. We will see some examples in this article. 

A topological field theory with values in vector spaces is a symmetric monoidal functor
\begin{align} \Cob(n) \to \Vect_K \end{align} from (some version of) the $n$-dimensional bordism category to the category of vector spaces over some field $K$ or, more generally, to any symmetric monoidal category. Results exhibiting the strong relation between topological field theories and algebraic objects include the classification of two-dimensional topological field theories by commutative Frobenius algebras, see \cite{kock}, and the classification of extended three-dimensional topological field theories by modular tensor categories, see \cite{BDSPV153D}.

In fact, there is the following natural generalization: For any finite group
$G$, we can consider a $G$-equivariant version $G\text{-}\Cob(n)$ of the bordism category, in which the bordisms are equipped with principal $G$-bundles.
A $G$-equivariant topological field theory is a symmetric monoidal functor $Z: G\text{-}\Cob(n) \to \Vect_K$ satisfying a homotopy invariance property, see \cite{turaevhqft} and also Section~\ref{mscchaphqfts} of the present article.
This gives a family of different types of equivariant topological field theories, one
for each finite group. It is a natural question whether a group homomorphism
$G\to H$ gives rise to a functor mapping $G$-equivariant to $H$-equivariant theories.
The present article answers this question affirmatively by providing a pushforward operation along a group morphism.

We work in the framework of monoidal
categories and non-extended topological field theories. However, the essential ideas should not depend on orientability and generalize to extended topological field theories. For our main result, no assumptions on the dimension are necessary.

We formulate in Section~\ref{secorbifoldconstruction} the pushforward operation along a morphism $\lambda : G \to H$ of finite groups by a two-step procedure:
\begin{xenumerate}
\item In the first step we produce from a $G$-equivariant topological field theory $Z$ an $H$-equivariant topological field theory $\widehat{Z}^\lambda : H\text{-}\Cob \to  \RepGrpd{K}$ with values in a symmetric monoidal category $\RepGrpd{K}$ defined in \cite{haugseng,trova} using flat vector bundles over essentially finite groupoids (alternatively, representations of essentially finite groupoids) and spans of groupoids. We refer to this step, which is described in Section~\ref{secchangetoequivcoeff}, as \emph{change to equivariant coefficients}. 
It relies crucially on the intertwining property following from the homotopy invariance axiom of homotopy quantum field theories (Proposition~\ref{satzfreehtpmaponbordmmsc}). 
As a byproduct, the change to equivariant coefficients allows us to make the relation between equivariant topological field theories in the sense of \cite{turaevhqft} and classical field theories with $G$-bundle background in the sense of \cite[3.]{fhlt} explicit, see Remark~\ref{bmkfhltcl}. 

\item To produce a theory with values in vector spaces, one wishes to take invariants,
see \cite[8.]{fhlt}. The corresponding symmetric monoidal functor $\Par : \RepGrpd{K} \to \Vect_K$ for a field $K$ of characteristic zero
has been provided in \cite{trova}.
This functor takes parallel sections of vector bundles over groupoids and will hence be referred to as \emph{parallel section functor}, see Section~\ref{secparsec} for a brief review.  

\end{xenumerate}
The pushforward operation along $\lambda : G \to H$ is the concatenation
 \begin{align}
\lambda_* : \HSym(G\text{-}\Cob(n),\Vect_K) \xrightarrow{\widehat{?}^\lambda} \HSym(H\text{-}\Cob(n) , \RepGrpd{K}) \xrightarrow{\Par_*} \HSym(H\text{-}\Cob(n), \Vect_K) .
\end{align}It is compatible with composition of group morphisms (Theorem~\ref{thmgluinglawpush}). 

The orbifoldization of $G$-equivariant topological field theories 
\begin{align}
\frac{?}{G}: \HSym(G\text{-}\Cob(n),\Vect_K) \to \Sym ( \Cob(n) , \Vect_K )
\end{align}
is defined as the pushforward along the morphism $G \to 1$ to the trivial group. 
Contrary to a possible guess, this functor cannot be obtained as an adjoint of the pullback along the forgetful functor $G\text{-}\Cob(n) \to \Cob(n)$ as explained in Example~\ref{extwisteddworbmsc}. Moreover, we show in Corollary~\ref{koresssurj} that the orbifold construction is essentially surjective. 

We can provide concrete formulae for the orbifold construction:
To a closed oriented $n$-dimensional manifold the orbifold theory $Z/G$ of a $G$-equivariant topological field theory $Z$ assigns the invariant
\begin{align}
\frac{Z}{G}(M) = \int_{\PBun_G(M)} Z(M,P)\,\td P, 
\end{align} where $Z(M,?)$ is an invariant function on the (finite) groupoid of $G$-bundles over $M$ that can be integrated with respect to groupoid cardinality (Corollary~\ref{korofkexplicit}, \ref{korofkexplicitc}). This implements an integral over all twisted sectors. Hence, at the level of partition functions orbifoldization is integration with respect to groupoid cardinality.

We then show that our orbifold construction encompasses the following instances of orbifoldization: 

\begin{itemize}\itemsep0pt

\item For two-dimensional $G$-equivariant theories a classification by $G$-crossed Frobenius algebras due to \cite{turaevhqft} is available, which can be substantially simplified using the change to equivariant coefficients (Example~\ref{exGcrossedproof}). We describe in Section~\ref{secmscofk2d} the orbifold construction in dimension two on the level of Frobenius objects by using the notion of an orbifold Frobenius algebra in the sense of \cite{kaufmannorb}.

\item The orbifold theory of the primitive homotopy quantum field theory twisted by a cocycle (in the sense of \cite{turaevhqft},~I.2.1) is the Dijkgraaf-Witten theory twisted by this cocycle (Example~\ref{extwisteddworbmsc}). The corresponding extended topological field theory provides the twisted Drinfeld double, see \cite{drp90}.

\item The orbifold theory of the $J$-equivariant Dijkgraaf-Witten theory $Z_\lambda$ constructed in \cite{maiernikolausschweigerteq} from a short exact sequence 
$0 \to G \to H \stackrel{\lambda}{\to} J \to 0$ of finite groups is isomorphic to the Dijkgraaf-Witten theory $Z_H$ for the group $H$ (Example~\ref{orbifolddwmodels}). Considered at the level of invariants this result is a particular incarnation of Cavalieri's principle for groupoids. 

\end{itemize}

The present article is also the basis for the orbifold construction for extended topological field theories which relates much richer algebraic structures.  
In particular, in the 3-2-1-dimensional case the orbifold construction should yield a geometric understanding of the algebraic orbifold construction in \cite{kirrolovg04} and hence allow for a study of equivariant modular categories and their orbifoldization in purely geometric terms. Conversely, we will be able to use the intimate algebraic relation of an equivariant modular category and its orbifoldization to understand extended equivariant topological field theories.

\section*{Acknowledgements}
We are grateful to Anssi Lahtinen, Lukas M\" uller and Alexis Virelizier for helpful discussions. 
In particular, we thank Alexis Virelizier for suggesting to us a generalization of the orbifold construction to a push operation along a group morphism and an anonymous referee for pointing out reference
\cite{trova}. These two suggestions have allowed to us to write this article in a 
much more concise and conceptual form.

CS is partially supported by the Collaborative Research Centre 676 ``Particles,
Strings and the Early Universe -- the Structure of Matter and Space-Time"
and by the RTG 1670 ``Mathematics inspired by String theory and Quantum
Field Theory''.
LW is supported by the RTG 1670 ``Mathematics inspired by String theory and Quantum
Field Theory''.

 \extremespace
\section{Homotopy quantum field theories and equivariant topological field theories\label{mscchaphqfts}}
In this introductory section we recall the definition of topological field theories with arbitrary target following \cite{turaevhqft} and explain how they give rise to representations of mapping groupoids.

\subsection{Topological field theories with arbitrary target space}
Topological field theories are symmetric monoidal functors on the bordism category $\Cob(n)$, see \cite{kock}. Replacing bordisms by bordisms together with a map into some fixed topological space $T$, yields a natural generalization of the bordism category, see \cite{turaevhqft}:

\begin{definition}[Bordism category for arbitrary target space]\label{defmscbordcattarget}
Let $n\ge 1$. For a non-empty topological space $T$ the \emph{category $T\text{-}\Cob(n)$ of $n$-dimensional bordisms carrying maps with target space $T$} is defined in the following way:
\begin{xenumerate}
\item Objects are pairs $(\Sigma,\varphi)$, where $\Sigma$ is an $n-1$-dimensional oriented closed manifold (hence an object in $\Cob(n)$) and $\varphi: \Sigma \to T$ a continuous map. By \emph{manifold} we always mean smooth finite-dimensional manifold. A continuous map will often just be referred to as map.
\item A morphism $(M,\psi) : (\Sigma_0,\varphi_0) \to (\Sigma_1,\varphi_1)$ is an equivalence class of pairs of oriented compact bordisms $M:\Sigma_0 \to \Sigma_1$ and continuous maps $\psi : M \to T$ such that the diagram of continuous maps
\begin{center}
\begin{tikzpicture}[scale=1]
\node (A1) at (0,0) {$\Sigma_0$};
\node (A2) at (2,1) {$M$};
\node (A3) at (4,0) {$\Sigma_1$};
\node (B2) at (2,-1) {$T$};
\node (B1) at (1,-0.5) {$$};
\node (B3) at (2,-1) {$$};
\path[<-,font=\scriptsize]
(A2) edge node[above]{$$} (A1)
(A2) edge node[above]{$$} (A3)
(B2) edge node[below]{$\varphi_0$} (A1)
(B2) edge node[below]{$\varphi_1$} (A3)
(B2) edge node[right]{$\psi$} (A2);
\end{tikzpicture}
\end{center} commutes. The unlabeled arrows are the embeddings of the boundary components into $M$. Two such pairs $(M,\psi)$ and $(M',\psi')$ are defined to be equivalent if there is an orientation-preserving diffeomorphism $\Phi : M \to M'$ making the diagram
\begin{center}
\begin{tikzpicture}[scale=1]
\node (A1) at (0,0) {$\Sigma_0$};
\node (A2) at (2,1) {$M$};
\node (A3) at (4,0) {$\Sigma_1$};
\node (B2) at (2,-1) {$M'$};
\node (B1) at (1,-0.5) {$$};
\node (B3) at (2,-1) {$$};
\path[<-,font=\scriptsize]
(A2) edge node[above]{$$} (A1)
(A2) edge node[above]{$$} (A3)
(B2) edge node[below]{$$} (A1)
(B2) edge node[below]{$$} (A3)
(B2) edge node[right]{$\Phi$} (A2);
\end{tikzpicture}
\end{center} commute such that additionally $\psi = \psi' \circ \Phi$.

\end{xenumerate}
The identity of $(\Sigma,\varphi)$ is represented by the cylinder over $\Sigma$ carrying the trivial homotopy $\varphi \simeq \varphi$. Composition is by gluing of bordisms and maps, respectively. Just like $\Cob(n)$, the category $T\text{-}\Cob(n)$ carries the structure of a symmetric monoidal category with duals. 
\end{definition}

\begin{definition}[Homotopy quantum field theory]\label{defmschqft}
An \emph{$n$-dimensional homotopy quantum field theory over a field $K$ with target space $T$} is a symmetric monoidal functor
\begin{align} Z : T\text{-}\Cob(n) \to \Vect_K\end{align} 
which is additionally \emph{homotopy invariant}, i.e.\ $Z(M,\psi_0) = Z(M,\psi_1)$ for morphisms $(M,\psi_0),(M,\psi_1)$ in $T\text{-}\Cob(n)$ with $\psi_0\simeq \psi_1$ relative $\partial M$.
For a group $G$ we define a \emph{$G$-equivariant topological field theory} as a homotopy quantum field theory with the classifying space $BG$ of $G$ as a target (hence we have an aspherical target). We set $G\text{-}\Cob(n) := BG\text{-}\Cob(n)$, so a $G$-equivariant topological theory is a symmetric monoidal functor
$Z : G\text{-}\Cob(n) \to \Vect_K$ fulfilling the homotopy invariance property. 
\end{definition}

\begin{remarks}\label{bmkgtftsmsc}

\item If $T$ is the one-point-space, then $T\text{-}\Cob(n)\cong \Cob(n)$ and we obtain an ordinary topological field theory.

\item Here and in the sequel the manifolds involved in the definition of $\Cob(n)$ or $T\text{-}\Cob(n)$ are always oriented. Therefore, one could specify all theories as \emph{oriented}. Since we only consider this case, we will drop this additional adjective.

\item If $M$ is a closed oriented $n$-dimensional manifold and $\varphi : M \to T$ a continuous map, then an $n$-dimensional topological field theory $Z : T\text{-}\Cob(n) \to \Vect_K$ assigns to $(M,\varphi)$ a number in $K$. This number is a diffeomorphism invariant of $M$. If $Z$ is a homotopy quantum field theory, it is also an invariant of the homotopy class of $\varphi$.

\item For any manifold $\Sigma$ the mapping groupoid $\Pi(\Sigma,BG)$ is canonically equivalent to the groupoid $\PBun_G(\Sigma)$ of principle fiber bundles over $\Sigma$ with structure group $G$, see e.g.\ \cite[1]{heinlothstacks}. Hence, we may view a $G$-equivariant topological field theory as being defined on a bordism category decorated with $G$-bundles.\label{bmkGequivtft2}

\end{remarks}

\begin{definition}\label{defcatofhqft}
For any target space $T$ we define the \emph{category $\HSym(T\text{-}\Cob(n),\Vect_K)$ of $n$-dimensional homotopy quantum field theories with target $T$} to be the full subcategory of $\Sym(T\text{-}\Cob(n),\Vect_K)$ of symmetric monoidal functors $T\text{-}\Cob(n)\to \Vect_K$ consisting of those functors having the homotopy invariance property from Definition~\ref{defmschqft}. 
\end{definition}

\noindent Since $T\text{-}\Cob(n)$ has duals and since monoidal natural transformations between symmetric monoidal functors on categories with duality are always isomorphisms we obtain the following:

\begin{corollary}[]\label{kormsccatTtfts}
For any target space $T$ the category $\HSym(T\text{-}\Cob(n),\Vect_K)$ of homotopy quantum field theories with target $T$ is a groupoid.
\end{corollary}

\begin{remark}
	\label{pointedunpoited}
In \cite{turaevhqft} a pointed version $T_*\text{-}\Cob(n)$ of the category $T\text{-}\Cob(n)$ is used for a fixed basepoint $t_0 \in T$: Objects are pairs of a closed oriented $n-1$-dimensional pointed manifold $\Sigma$ (\emph{pointed} means that all components are equipped with a basepoint) and a pointed map $\varphi : \Sigma \to T$. The morphisms, their composition and the monoidal structure are defined just as in $T\text{-}\Cob(n)$, i.e.\ the forgetful functor
$
U: T_*\text{-}\Cob(n) \to T\text{-}\Cob(n)
$ is a symmetric monoidal equivalence for path-connected $T$ as can be easily seen. 
This shows that the category of pointed homotopy quantum field theories with target space $T$ in the sense of \cite{turaevhqft} is equivalent to the category of homotopy quantum field theories with target space $T$ used in this article if $T$ is path-connected. 

\end{remark}

\subsection{Groupoid representations from evaluation on the cylinder}
\noindent By evaluation on the cylinder homotopy quantum field theories produce representations of mapping groupoids on vector spaces.

\begin{definition}
For topological spaces $X$ and $Y$ we denote by $\Pi(X,Y)$ the \emph{mapping groupoid of $X$ and $Y$} which has continuous maps $X \to Y$ as objects and equivalence classes of homotopies of such maps as morphisms. Here, we consider homotopies $h, h' : X\times [0,1] \to Y$ between $f$ and $g$ to be equivalent if they are homotopic relative the boundary $X\times \{0,1\}$ of the cylinder $X\times [0,1]$. The inverse of a homotopy $h$ will be denoted by $h^-$.
\end{definition}

\noindent The axiom of homotopy invariance ensures that we obtain representations of mapping groupoids and leads to an intertwining property which is absolutely crucial for the change to equivariant coefficients in Section~\ref{secchangetoequivcoeff}:

\begin{proposition}[]\label{satzmscreprpigrpd}
A homotopy quantum field theory $Z: T\text{-}\Cob(n) \to \Vect_K$ provides for each object $\Sigma$ in $\Cob(n)$ a representation 
\begin{align}
\varrho_\Sigma : \Pi(\Sigma,T) &\to \Vect_K \\ (\varphi:\Sigma \to T)& \mapsto Z(\Sigma,\varphi),\\ (\varphi \stackrel{h}{\simeq} \psi) & \mapsto Z(\Sigma\times [0,1] , h). 
\end{align} of the mapping groupoid $\Pi(\Sigma,T)$. 
\end{proposition}

\spaceplease
\begin{proposition}[Intertwining property]\label{satzfreehtpmaponbordmmsc}
Let $Z: T\text{-}\Cob(n) \to \Vect_K$ be a homotopy quantum field theory with target $T$. Given two morphisms $(M,\psi)$ and $(M,\xi)$ in $T\text{-}\Cob(n)$ and a homotopy $\psi \stackrel{h}{\simeq} \xi : M \to T$, the square
\begin{center}
\begin{tikzpicture}[scale=2]
\node (A1) at (0,1) {$Z(\Sigma_0,\psi|_{\Sigma_0})$};
\node (A2) at (4,1) {$Z(\Sigma_0,\xi|_{\Sigma_0})$};
\node (B1) at (0,0) {$Z(\Sigma_1,\psi|_{\Sigma_1})$};
\node (B2) at (4,0) {$Z(\Sigma_1,\xi|_{\Sigma_1})$};
\path[->,font=\scriptsize]
(A1) edge node[above]{$Z(\Sigma_0\times [0,1],h|_{\Sigma_0})=\varrho_{\Sigma_0}(h|_{\Sigma_0})$} (A2)
(A1) edge node[left]{$Z(M,\psi)$} (B1)
(B1) edge node[above]{$Z(\Sigma_1\times [0,1],h|_{\Sigma_1}) =\varrho_{\Sigma_1}(h|_{\Sigma_1})$} (B2)
(A2) edge node[right]{$Z(M,\xi)$} (B2);
\end{tikzpicture}
\end{center} commutes. 
\end{proposition}

\begin{proof}
We construct a homotopy $H: M\times [0,1] \to T$ relative $\partial M$ starting at $\psi$. In order to define the maps $H_t : M \to T$, we choose $(\Sigma_1 \times [0,1]) \cup_{\Sigma_1} M \cup_{\Sigma_1} (\Sigma_0 \times [0,1])$ as a representative for the bordism class $M$, where $\cup_?$ denotes the gluing of manifolds (we are gluing in two cylinders). Now let $H_t$ be the map $(\Sigma_1 \times [0,1]) \cup_{\Sigma_1} M \cup_{\Sigma_0} (\Sigma_0 \times [0,1])\to T$ obtained by gluing together the maps
\begin{align}
\alpha_t : \Sigma_0 \times [0,1] &\to T, \quad (x,s) \mapsto h(x,st),\\
h_t : M  &\to T, \quad (x,s) \mapsto h(x,t),\\
\beta_t : \Sigma_1 \times [0,1] &\to T, \quad (x,s) \mapsto h(x,(1-s)t).
\end{align} Hence, $H$ is a homotopy relative boundary from $\psi$ to the map \begin{align} M\cong (\Sigma_1 \times [0,1]) \cup_{\Sigma_1} M \cup_{\Sigma_0} (\Sigma_0 \times [0,1]) \to T\end{align} obtained by gluing $h|_{\Sigma_1}^{-}$, $\xi$ and $h|_{\Sigma_0}$. Using homotopy invariance and functoriality, we obtain
\begin{align}
Z(M,\psi) = Z(\Sigma_1\times [0,1],h|_{\Sigma_1}^{-}) \circ Z(M,\xi) \circ Z(\Sigma_0\times [0,1],h|_{\Sigma_0}).
\end{align} From Proposition~\ref{satzmscreprpigrpd} we deduce that $Z(\Sigma_1\times [0,1],h|_{\Sigma_1}^{-})=Z(\Sigma_1\times [0,1],h|_{\Sigma_1})^{-1}$.   
\end{proof}

\extremespace
\section{Change to equivariant coefficients and the parallel section functors}
We will now discuss the ingredients needed for the definition of the pushforward construction.

\subsection{Parallel sections of vector bundles over a groupoid\label{limitscolimitsgrpdreprmsc}}
Let $\Gamma$ be a small groupoid and $\varrho: \Gamma \to \Vect_K$ a representation. Then its limit and colimit can be expressed by
\begin{align}
\lim \varrho \cong \prod _{[x]\in \pi_0(\Gamma)}\varrho(x)^{\Aut(x)}, \quad 
\colim \varrho \cong \bigoplus _{[x]\in \pi_0(\Gamma)} \varrho(x)_{\Aut(x)}, \label{eqncolimlim}
\end{align} i.e.\ by invariants and coinvariants, respectively. There is a canonical map $\lim \varrho \to \colim \varrho$. If $\Gamma$ is essentially finite (Definition~\ref{defessentiallyfinitegroupoid}) and $K$ of characteristic zero, this map is an isomorphism.

The formulae in \eqref{eqncolimlim} are very explicit, but not too useful since they use chosen representatives of the isomorphism classes. 
We will now discuss a very convenient realization of the limit of a representation by seeing a groupoid representation as a vector bundle over a groupoid (see Proposition~\ref{satzholonomyprinciplemsc} below). All the notions occuring in the following definition are directly transferred from the ordinary theory of vector bundles with connection. They also appear in the context of groupoid representations in \cite{willterongerbesgrpds}.

\begin{definition}[Vector bundle over a groupoid]\label{defmscvectorbundleongroupoid}
Let $K$ be a field and $\Gamma$ a small groupoid. We will view a representation $\varrho : \Gamma \to \FinVect_K$ of $\Gamma$ on finite-dimensional $K$-vector spaces as a \emph{$K$-vector bundle $\varrho$ over $\Gamma$ (with flat connection)}. We call the vector space $\varrho(x)$ the \emph{fibre} over $x\in \Gamma$. We will not define directly what a connection on $\varrho$ is, but we will define its parallel transport: For a morphism $g : x \to y$ in $\Gamma$ we call the operator $\varrho(g):\varrho(x) \to \varrho(y)$ the \emph{parallel transport} of $\varrho$ along $g$. 
A morphism $\eta : \varrho \to \xi$ of vector bundles over $\Gamma$ is a natural transformation of the corresponding functors $\Gamma \to \Vect_K$ (this means that $\eta$ is an intertwiner for the parallel transport operators). The category of $K$-vector bundles over $\Gamma$ is denoted by $\VecBun_K(\Gamma)$. 
\end{definition}

The category $\VecBun_K(\Gamma)$ of $K$-vector bundles over some groupoid $\Gamma$ inherits from $\FinVect_K$ the structure of a symmetric monoidal category with duals. The monoidal unit $\mathbb{I}_\Gamma$ assigns to every $x\in \Gamma$ the vector space $K$ and to every morphism in $\Gamma$ the identity on $K$.

For a vector bundle over a groupoid there is the notion of a parallel section, sometimes also called \emph{invariant} or \emph{flat sections}, see \cite{willterongerbesgrpds}. The relation to parallel sections in the geometric sense is obvious.

\begin{definition}[Parallel section of a vector bundle over a groupoid]\label{defparallelsectionsmsc}
Let $\varrho$ be a $K$-vector bundle over a groupoid $\Gamma$. A \emph{parallel section} of $\varrho$ is a function $s$ on $\Gamma$ with $s(x) \in \varrho(x)$ for $x\in \Gamma$ such that for any morphism $g : x \to y$ the equation
$
s(y) = \varrho(g)s(x)
$ holds. By \begin{align} \Par \varrho := \{s : \Gamma \to K \, | \,  s \ \text{parallel section}\} \end{align} we denote the \emph{vector space of parallel sections of $\varrho$}. 
\end{definition}

\noindent The following result is the analogue of the well-known \emph{holonomy principle} from the theory of vector bundles with connection.

\begin{proposition}[Holonomy principle]\label{satzholonomyprinciplemsc}
Let $\varrho$ be a $K$-vector bundle over a small groupoid $\Gamma$. Then the vector space $\Par \varrho$ of parallel sections of $\varrho$ is the limit of $\varrho$. By the functoriality of the limit, taking parallel sections extends to a functor
$
\Par_\Gamma: \VecBun_K(\Gamma) \to \Vect_K.
$
\end{proposition}

Let $\lambda : \varrho \to \xi$ be a morphism of $K$-vector bundles over a small groupoid $\Gamma$. Then the image of $\lambda$ under $\Par_\Gamma$ will be denoted by  
\begin{align}
\lambda_* : \Par \varrho \to \Par \xi, \quad s \mapsto \lambda_*s,
\end{align} and is explicitly given by $(\lambda_* s)(x) := \lambda_x s(x)$ for $x\in \Gamma$.

Just like vector bundles over topological spaces, vector bundles over groupoids admit pullbacks satisfying some obvious properties:

\spaceplease
\begin{proposition}[Pullback of vector bundles over groupoids]\label{pullbackvbgrpdmsc}
Let $\Phi: \Gamma \to \Omega$ be a functor between small groupoids and $\varrho$ a $K$-vector bundle over $\Omega$. Then $\Phi^* \varrho:= \varrho \circ \Phi : \Gamma \to \Vect_K$ is a vector bundle over $\Gamma$. This provides a pullback functor
\begin{align}
\Phi^* : \VecBun_K(\Omega) \to \VecBun_K(\Gamma)
\end{align} 
\begin{myenumerate}
\item The pullback functors obey the composition law $(\Psi \circ \Phi)^* = \Phi^* \circ \Psi^*$, where $\Psi: \Omega \to \Lambda$ is another functor between small groupoids. \label{pullbackvbgrpdmsca}

\item Let $\Phi': \Gamma \to \Omega$ be another functor between small groupoids and $\eta : \Phi \Rightarrow \Phi'$ a natural isomorphism. Then $\eta$ induces an isomorphism $\varrho(\eta) : \Phi^* \varrho \to {\Phi'}^* \varrho$ consisting of the maps $(\varrho(\eta_x))_{x\in\Gamma}$. \label{pullbackvbgrpdmscb}

\item The functor $\Phi$ induces a linear map
\begin{align}
\Phi^* : \Par \varrho \to \Par \Phi^* \varrho, \quad s \mapsto \Phi^* s,
\end{align} where $(\Phi^* s)(x) = s (\Phi(x))$ for all $x\in \Gamma$. Such a map is called a pullback map. \label{pullbackvbgrpdmscc}

\item Just like the pullback functors the pullback maps obey the composition law $(\Psi \circ \Phi)^* = \Phi^* \circ \Psi^*$.\label{pullbackvbgrpdmscd}

\item The pullback maps are natural in the sense that they provide a natural transformation
\begin{align}
\Par_\Omega \to \Par_\Gamma \circ\  \Phi^*, 
\end{align} i.e.\ for any morphism $\lambda : \varrho \to \xi$ of vector bundles over $\Omega$ the square 
\begin{center}
\begin{tikzpicture}[scale=2]
\node (A1) at (0,1) {$\Par \varrho$};
\node (A2) at (1,1) {$\Par \Phi^* \varrho$};
%\node (A5) at (4,1) {$E$};
\node (B1) at (0,0) {$\Par \xi$};
\node (B2) at (1,0) {$\Par \Phi^* \xi$};
\path[->,font=\scriptsize]
(A1) edge node[above]{$\Phi^*$} (A2)
(A1) edge node[left]{$\lambda_*$} (B1)
%(A4) edge node[above]{$i$} (A5)
(B1) edge node[above]{$\Phi^*$} (B2)
(A2) edge node[right]{$(\Phi^* \lambda)_*$} (B2);
\end{tikzpicture}
\end{center}
commutes.\label{pullbackvbgrpdmsce}

\label{pullbackvbgrpdmscc}
\end{myenumerate}
\end{proposition}

\subsection{The symmetric monoidal category $\RepGrpd{K}$\label{secrepgrpdmsc}}
The symmetric monoidal category $\RepGrpd{K}$ of vector bundles over different groupoids can be seen as an equivariant version of the span category in \cite{mortonvec,morton1}. Its definition is sketched in \cite{fhlt} and worked out in \cite{haugseng,trova}.

In order to define $\RepGrpd{K}$ we need the notion of a homotopy pullback of a cospan of groupoids. 
A \emph{span of groupoids} is a diagram $\Gamma \longleftarrow \Lambda \to \Omega$ of groupoids and functors between them. Dually, a \emph{cospan of groupoids} is a diagram $\Gamma \to \Lambda \longleftarrow \Omega$ of groupoids and functors between them.

\begin{definition}\label{defhomotopypullbackofk}
To every cospan $\Gamma \stackrel{\Phi}{\longrightarrow} \Omega \stackrel{\Psi}{\longleftarrow} \Lambda$ of groupoids we can associate a groupoid $\Gamma \times_{\Omega} \Lambda$, called the \emph{homotopy pullback} or \emph{weak pullback} of $\Gamma \stackrel{\Phi}{\longrightarrow} \Omega \stackrel{\Psi}{\longleftarrow} \Lambda$, which is the groupoid of triples $(x,y,\eta_{x,y})$, where $x\in \Gamma$, $y\in \Lambda$ and $\Phi(x)\stackrel{\eta_{x,y}}{\cong} \Psi(y)$.  We have the obvious projection functors
$\pi_\Gamma : \Gamma \times_\Omega \Lambda \to \Gamma$ and 
$\pi_{\Lambda} : \Gamma \times_\Omega \Lambda \to \Lambda$. 
The assignment $\Gamma \times_{\Omega} \Lambda\ni(x,y,\eta_{x,y}) \mapsto \eta_{x,y}$ defines a natural isomorphism $\Phi \circ \pi_\Gamma \Rightarrow \Psi \circ \pi_\Lambda$, i.e.\ the square\begin{center}
\begin{tikzpicture}[scale=2, implies/.style={double,double equal sign distance,-implies},
     dot/.style={shape=circle,fill=black,minimum size=2pt,
                 inner sep=0pt,outer sep=0pt},]
\node (A1) at (0,1) {$\Gamma\times_\Omega \Lambda$};
\node (A2) at (1,1) {$\Gamma$};
\node (B1) at (0,0) {$\Lambda$};
\node (B2) at (1,0) {$\Omega$};
\path[->,font=\scriptsize]
(A1) edge node[above]{$\pi_\Gamma$} (A2)
(A1) edge node[left]{$\pi_\Lambda$} (B1)
(A2) edge node[right]{$\Phi$} (B2)
(B1) edge node[below]{$\Psi$} (B2);
\draw (A2) edge[implies] node[above] {\scriptsize$\eta\ $} (B1);
\end{tikzpicture}
\end{center} weakly commutes. In the case $\Lambda = \star$ the map $\star \to \Omega$ specifies an object $y \in \Omega$, and the resulting homotopy pullback is called the \emph{homotopy fiber $\Phi^{-1}[y]$ of $\Phi$ over $y$}. 
\end{definition}

\begin{definition}[The category $\RepGrpd{K}$]\label{defmscRepGrpd}
For a field $K$ define the category $\RepGrpd{K}$ as follows: \begin{myenumerate}
\item Objects are pairs $(\Gamma,\varrho)$, where $\varrho : \Gamma \to \FinVect_K$ is a vector bundle over an essentially finite groupoid $\Gamma$.

\item A morphism $(\Gamma,\varrho) \to (\Omega,\xi)$ is a equivalence class of pairs $(\Lambda, \lambda)$, where $\Gamma \stackrel{r_0}{\longleftarrow} \Lambda \stackrel{r_1}{\to} \Omega$ is a span of essentially finite groupoids and $\lambda : r_0^* \varrho \to r_1^* \xi$ is a morphism in $\VecBun_K(\Lambda)$.

\item The composition of the morphisms $(\Gamma,\rho) \stackrel{r_0}{\longleftarrow} (\Lambda,\lambda) \stackrel{r_1}{\to} (\Omega,\xi)$ and $(\Omega,\xi) \stackrel{r_1'}{\longleftarrow} (\Lambda',\lambda') \stackrel{r_2'}{\to} (\Xi,\nu)$ is defined to be equivalence class of the span $\Gamma \longleftarrow \Lambda \times_\Omega \Lambda' \to \Omega$ defined by the diagram
\begin{center}
\begin{tikzpicture}[scale=1,     implies/.style={double,double equal sign distance,-implies},
     dot/.style={shape=circle,fill=black,minimum size=2pt,
                 inner sep=0pt,outer sep=2pt},]
\node (A1) at (0,0) {$\Gamma$};
\node (A2) at (2,2) {$\Lambda$};
\node (A3) at (4,0) {$\Omega$};
\node (C) at (4,0.2) {$$};
\node (A4) at (6,2) {$\Lambda'$};
\node (A5) at (8,0) {$\Xi$};
%\node (B1) at (2,2) {$M\setminus \Sigma_1$};
\node (B2) at (4,4) {$\Lambda \times_\Omega \Lambda'$};
%\node (B3) at (6,2) {$M'\setminus\Sigma_1$};
\path[->,font=\scriptsize]
(A2) edge node[above]{$r_0$} (A1)
(A2) edge node[above]{$r_1$} (A3)
(A4) edge node[above]{$r_1'\ $} (A3)
(A4) edge node[above]{$r_2'$} (A5)
(B2) edge node[left]{$p$} (A2)
(B2) edge node[right]{$\ p'$} (A4);
\draw (A2) edge[implies] node[above] {\scriptsize$\eta$} (A4);
\end{tikzpicture},
\end{center} where $\Lambda \times_\Omega \Lambda'$ is the weak pullback coming with the projections $p$ and $p'$ and the natural isomorphism $\eta : r_1 \circ p \Rightarrow r_1' \circ p'$. The needed morphism $\lambda' \times_\Omega \lambda :(r_0\circ p)^* \varrho \to (r_2 \circ p')^* \nu$ is defined as the composition
\begin{align}
(r_0\circ p)^* \varrho = p^* r_0^* \varrho \xrightarrow{p^* \lambda} p^* r_1^* \xi=(r_1 \circ p)^* \xi \xrightarrow{\xi(\eta)} (r_1'\circ p')^* \xi = {p'}^* {r_1'}^* \xi\xrightarrow{{p'}^* \lambda'} {p'}^* {r_2'}^* \nu=(r_2\circ p')^* \nu. 
\end{align}
Here $\xi(\eta) : (r_1 \circ p)^* \xi \to (r_1'\circ p')^* \xi$ is the morphism of vector bundles induced from $\eta : r_1 \circ p \Rightarrow r_1' \circ p'$ according to Proposition~\ref{pullbackvbgrpdmsc}, \ref{pullbackvbgrpdmscb}.

\end{myenumerate}

\end{definition}

\noindent The category $\RepGrpd{K}$ carries in a natural way the structure of a symmetric monoidal category. The tensor product is analogous to the external product known from $K$-theory, see \cite{hatcherktheorie}. 
The duality is inherited from $\FinVect_K$, see \cite{haugseng,trova}:

\begin{proposition}[]\label{satzmscdualizingreps}
The symmetric monoidal category $\RepGrpd{K}$ has coinciding left and right duals. 
\end{proposition}

\subsection{Change to equivariant coefficients\label{secchangetoequivcoeff}}
Given a morphism $\lambda : G \to H$ of finite groups we will produce, by a procedure which will be referred to as \emph{change to equivariant coefficients}, from $G$-equivariant topological field theory $Z: G\text{-}\Cob(n) \to \Vect_K$ an $H$-equivariant topological field theory $\widehat{Z}^\lambda: H\text{-}\Cob(n) \to \RepGrpd{K}$ taking values in $\RepGrpd{K}$.

First recall that for a discrete group $G$, the functor $\Pi(?,BG)$ provides a model for the stack of $G$-bundles, see Remark~\ref{bmkgtftsmsc},~\ref{bmkGequivtft2}.
Let us list those properties of this stack relevant for our construction:

\begin{myenumerate}

\item Homotopy invariance: A homotopy $f \stackrel{h}{\simeq} g$ of maps $f,g: \Sigma \to \Sigma'$ gives rise to a natural isomorphism
\begin{align}
f^* \cong g^* : \Pi(\Sigma',BG) \to \Pi(\Sigma,BG)
\end{align} of the corresponding pullback functors. 

\item Additivity: For all manifolds $\Sigma$ and $\Sigma'$ the inclusions $\iota : \Sigma \to \Sigma \coprod \Sigma'$ and $\iota' : \Sigma' \to \Sigma \coprod \Sigma'$ induce an equivalence
\begin{align}
\Pi\left(\Sigma \coprod \Sigma,BG\right) \xrightarrow{\iota^* \times {\iota'}^*} \Pi(\Sigma,BG) \times \Pi(\Sigma',BG),
\end{align} and  $\Pi(\emptyset,BG)$ is naturally equivalent to the trivial groupoid with one object.

\item Finiteness: If $G$ is finite, then for every compact manifold $K$ (with boundary) the groupoid $\Pi(K,BG)$ is essentially finite, see \cite{morton1}.

\item Gluing property (with respect to bordisms): For morphisms $M: \Sigma_0 \to \Sigma_1$ and $M': \Sigma_1 \to \Sigma_2$ in $\Cob(n)$ for $n\ge 1$ the inclusions $j: M\to M'\circ M$ and $j' : M' \to M'\circ M$ induce an equivalence of groupoids
\begin{align}
\Pi(M'\circ M,BG) \xrightarrow{j^* \times {j'}^*} \Pi(M,BG)\times_{\Pi(\Sigma_1,BG)} \Pi(M',BG),
\end{align} where $\Pi(M,BG)\times_{\Pi(\Sigma_1,BG)} \Pi(M',BG)$ is the homotopy pullback of the cospan $\Pi(M,BG) \to \Pi(\Sigma_1,BG) \longleftarrow \Pi(M',BG)$. 
The gluing property follows from the fact that $\Pi(?,BG)$ is a stack and implies additivity. 
\end{myenumerate}

Now let $\lambda : G \to H$ be a morphism of finite groups and 
\begin{align}
\lambda_* : \Pi(?,BG) \to \Pi (?,BH)
\end{align} the induced stack morphism. From a $G$-equivariant topological field theory $Z: G\text{-}\Cob(n) \to \Vect_K$ we would like to construct a homotopy invariant symmetric monoidal functor $\widehat{Z}^\lambda : H\text{-}\Cob(n) \to \RepGrpd{K}$:
\begin{itemize}
\item To an object $(\Sigma,\varphi)$ in $H\text{-}\Cob(n)$ we assign the homotopy fiber $\lambda_*^{-1}[\varphi]$ defined by the homotopy pullback square
\begin{center}
\begin{tikzpicture}[scale=1.5, implies/.style={double,double equal sign distance,-implies},
dot/.style={shape=circle,fill=black,minimum size=2pt,
	inner sep=0pt,outer sep=2pt},]
\node (A1) at (0,1) {$\lambda_*^{-1}[\varphi]$};
\node (A2) at (2,1) {$\Pi(\Sigma,BG)$};
\node (B1) at (0,0) {$\star$};
\node (B2) at (2,0) {$\Pi (\Sigma,BH)$};
\path[->,font=\scriptsize]
(A1) edge node[above]{$q$} (A2)
(A1) edge node[left]{$$} (B1)
(A2) edge node[right]{$\lambda_*$} (B2)
(B1) edge node[below]{$\varphi$} (B2);
\end{tikzpicture}, 
\end{center} and the pullback vector bundle $q^* \varrho_\Sigma : \lambda_*^{-1}[\varphi] \to \Vect_K$, where $\varrho_\Sigma$ is the vector bundle that $Z$ gives rise to in the sense of Proposition~\ref{satzmscreprpigrpd}. 

\item To a morphism $(M,\psi) : (\Sigma_0,\varphi_0) \to (\Sigma_1,\varphi_1)$ in $H\text{-}\Cob(n)$ we assign \begin{itemize} \item the span
\begin{align}
\lambda_*^{-1}[\varphi_0] \stackrel{r_0}{\longleftarrow} \lambda^{-1}_*[\psi ] \stackrel{r_1}{\to} 	\lambda_*^{-1}[\varphi_1],
\end{align} in which the needed functors are induced by restriction, i.e.\ the functors
\begin{align}
q_0 : \lambda_*^{-1}[\varphi_0] & \to \Pi(\Sigma_0,BG),\\
q : \lambda_*^{-1}[\psi] & \to \Pi(M,BG),\\
q_1 : \lambda_*^{-1}[\varphi_0] & \to \Pi(\Sigma_1,BG)\end{align}
and the restriction functors
\begin{align}
\Pi(\Sigma_0,BG)  \stackrel{s_0}{\longleftarrow} \Pi(M,BG) \stackrel{s_1}{\to}    \Pi(\Sigma_1,BG) 
\end{align} fulfill $q_0 r_0=s_0q$ and $q_1 r_1 = s_1 q$,

\item and the intertwiner 
\begin{align} q^* Z(M,?) : r_0^* q_0^* \varrho_{\Sigma_0} = q^* s_0^* \varrho_{\Sigma_0} \to q^* s_1^* \varrho_{\Sigma_1} = r_1^* q_1^* \varrho_{\Sigma_1} \end{align} obtained as the pullback of the intertwiner $Z(M,?) : s_0^* \varrho_{\Sigma_0} \to s_1^* \varrho_{\Sigma_1}$. The fact that we actually get an intertwiner follows from Proposition~\ref{satzfreehtpmaponbordmmsc} (note that the homotopy invariance crucially enters here).

\end{itemize}

\end{itemize}

\begin{theorem}[Change to equivariant coefficients functor]\label{thmchangetoequivcoeff}
	For any morphism $\lambda : G \to H$ of finite groups the assignment $Z \mapsto \widehat{Z}^\lambda$ extends to a functor
	\begin{align} \widehat{?}^\lambda : \HSym(G\text{-}\Cob(n), \Vect_K) \to \HSym(H\text{-}\Cob(n) , \RepGrpd{K}).\end{align} 
\end{theorem}

\begin{proof}
	We will only prove that for a given $G$-equivariant topological field theory $Z: G\text{-}\Cob(n) \to \Vect_K$ the definitions above make $\widehat{Z}^\lambda$ into a homotopy invariant symmetric monoidal functor from $H\text{-}\Cob(n)$ to $\RepGrpd{K}$. Once this is established, the functoriality of $\widehat{?}^\lambda$ will be clear. 
	
	\begin{pnum}
		\item Clearly, $\widehat{Z}^\lambda$ respects identities, so we turn directly to the gluing law: For morphisms $(M,\psi) : (\Sigma_0,\varphi_0) \to (\Sigma_1,\varphi_1)$ and $(M',\psi') : (\Sigma_1,\varphi_1) \to (\Sigma_2,\varphi_2)$ in $H\text{-}\Cob(n)$ the span part of the composition of $\widehat{Z}^\lambda(M,\psi)$ and $\widehat{Z}^\lambda(M',\psi')$ in $\RepGrpd{K}$ is the homotopy pullback of 
		\begin{align}
		\lambda_*^{-1}[\psi] \stackrel{r_1}{\to} \lambda^{-1}_*[\varphi_1 ] \stackrel{r_1'}{\longleftarrow} 	\lambda_*^{-1}[\psi'],
		\end{align} which by the universal property of the homotopy pullback and the gluing property of the stacks $\Pi(?,BG)$ and $\Pi(?,BH)$ can be seen to be naturally equivalent to $\lambda_*^{-1}[\psi \cup \psi']$, where $\psi \cup \psi' : M'\circ M \to BG$ is the map that $\psi$ and $\psi'$ give rise to since they coincide on $\Sigma_1$. Next we observe that the span
		\begin{align}
		\lambda_*^{-1}[\psi] \longleftarrow \lambda_*^{-1}[\psi \cup \psi'] \to \lambda_*^{-1}[\psi']
		\end{align} coming from restriction is the span part of the image of the composition of $(M,\psi)$ and $(M,\psi')$ under $\widehat{Z}^\lambda$. The gluing law can now be directly verified.
		\item The monoidal structure and the symmetry requirement is defined and checked, respectively, using the additivity of the stacks $\Pi(?,BG)$ and $\Pi(?,BH)$.

		\item For the proof of the homotopy invariance consider two morphisms $(M,\psi) , (M,\psi') : (\Sigma_0,\varphi_0) \to (\Sigma_1, \varphi_1)$. A homotopy $\psi \simeq \psi'$ relative $\partial M$ induces an isomorphism $\lambda_*^{-1}[\psi] \to \lambda_*^{-1}[\psi']$ making the diagram
		\begin{center}
			\begin{tikzpicture}[scale=1,     implies/.style={double,double equal sign distance,-implies},
			dot/.style={shape=circle,fill=black,minimum size=2pt,
				inner sep=0pt,outer sep=2pt},]
			\node (A1) at (0,0) {$\lambda^{-1}_*[\varphi_0]$};
			\node (A2) at (2,1) {$\lambda^{-1}_*[\psi]$};
			\node (A3) at (4,0) {$\lambda^{-1}_*[\varphi_1]$};
			\node (B2) at (2,-1) {$\lambda^{-1}_*[\psi']$};
			\node (B1) at (1,-0.5) {$$};
			\node (B3) at (3,-0.5) {$$};
			%\node (B2) at (1,0) {$B\Omega$};
			\path[->,font=\scriptsize]
			(A2) edge node[above]{$$} (A1)
			(A2) edge node[above]{$$} (A3)
			(B2) edge node[below]{$$} (A1)
			(B2) edge node[below]{$$} (A3)
			%(A4) edge node[above]{$i$} (A5)
			(A2) edge node[right]{$$} (B2);
			\end{tikzpicture}
			\end{center} commute. This together with the homotopy invariance of $Z$ implies the homotopy invariance of $\widehat{Z}^\lambda$.    
			
			\end{pnum}
			
			\end{proof}

\begin{remark}\label{bmkfhltcl}
In the special case $H=1$ Theorem~\ref{thmchangetoequivcoeff} allows us to obtain from a $G$-equivariant topological field theory $Z$ an ordinary topological field theory $\widehat{Z} : \Cob (n) \to \RepGrpd{K}$ making the triangle 
\begin{center}
	\begin{tikzpicture}[scale=1,     implies/.style={double,double equal sign distance,-implies},
dot/.style={shape=circle,fill=black,minimum size=2pt,
	inner sep=0pt,outer sep=2pt},]
\node (A1) at (0,0) {$\Cob(n)$};
\node (A2) at (4,0) {$\RepGrpd{K}$};
\node (A3) at (2,-2) {$\SpanGrpd$};
%\node (B2) at (1,0) {$B\Omega$};
\path[->,font=\scriptsize]
(A1) edge node[above]{$\widehat{Z}$} (A2)
(A1) edge node[left]{$\Pi(?,BG)$} (A3)
(A2) edge node[right]{$U$} (A3);
\end{tikzpicture}
\end{center} commute, where $U : \RepGrpd{K} \to \SpanGrpd$ is the forgetful functor to the category of spans of groupoids. Hence, in a special case, the change to equivariant coefficients allows us to relate homotopy quantum field theories to such commuting triangles which in \cite{fhlt} are referred to as a \emph{classical field theory with $G$-bundle background}. 
\end{remark}

\begin{example}[Crossed Frobenius $G$-algebras]\label{exGcrossedproof}
The change of coefficients offers an elegant and direct approach to the notion of a crossed Frobenius $G$-algebra as appearing in \cite[II,~3.2]{turaevhqft}:
For this note that \cite[Theorem~3.6.19]{kock} applied to the symmetric monoidal category $\RepGrpd{K}$ yields an equivalence 
\begin{align}
\Sym(\Cob(2),\RepGrpd{K}) \cong \cFrob(\RepGrpd{K})
\end{align} between the group\-o\-id of two-dimensional $\RepGrpd{K}$-valued topological field theories and the group\-oid of commutative Frobenius algebras in $\RepGrpd{K}$. 
Concatenating with the change of coefficients functor for the morphism $G\to 1$ from $G$ to the trivial group we obtain a functor
\begin{align} \HSym(G\text{-}\Cob(2),\Vect_K) \to \Sym(\Cob(2),\RepGrpd{K}) \cong \cFrob(\RepGrpd{K}) \label{exGcrossedproof1} \end{align}  assigning to a two-dimensional $G$-equivariant topological field theory $Z: G\text{-}\Cob(2)\to \Vect_K$ a commutative Frobenius algebra $\mathfrak{A}$ in the symmetric monoidal category $\RepGrpd{K}$. We claim that a commutative Frobenius algebra $\mathfrak{A}$ in $\RepGrpd{K}$ amounts precisely to the following data: $\mathfrak{A}$ has an underyling functor
\begin{align}
\varrho_\mathfrak{A} : \Pi(\sphere^1,BG)\cong G//G \to \Vect_K.
\end{align} By $\mathfrak{A}$ we also denote the direct sum $\mathfrak{A} = \bigoplus_{g\in G} \mathfrak{A}_g$, where $\mathfrak{A}_g := \varrho_\mathfrak{A}(g)$ for an object $g\in G$. For $g\in G$ and $v\in \mathfrak{A}_h$ we use the abbreviation $g.v:= \varrho_\mathfrak{A}(g)v\in \mathfrak{A}_{ghg^{-1}}$. We see $\mathfrak{A}$ and hence also $\mathfrak{A} \otimes \mathfrak{A}$ as a $G$-representation. 
We can specify the Frobenius structure by giving the associative and unital multiplication 
\begin{align}
\mu : \mathfrak{A} \otimes \mathfrak{A} \to \mathfrak{A} , \quad v \otimes w \mapsto vw
\end{align}
and the non-denegerate pairing 
\begin{align}
\kappa : \mathfrak{A} \otimes \mathfrak{A} \to K;
\end{align}
for a detailed account on the equivalent ways to describe Frobenius structures see \cite{fuchsstignerfrobenius}. They satisfy:
\begin{myenumerate}
\item The multiplication carries $\mathfrak{A}_g \otimes \mathfrak{A}_h$ to $\mathfrak{A}_{gh}$, has unit in $\mathfrak{A}_1$ and intertwines with the $G$-action. \label{exGcrossedproofa}

\item The pairing is $G$-invariant and satisfies $\kappa|_{\mathfrak{A}_g \otimes \mathfrak{A}_h}=0$ if $h\neq g^{-1}$. 

\end{myenumerate}
This summarizes the Frobenius structure on $\mathfrak{A}$. 
\begin{myenumerate}
\item[{\normalfont (c)}] The commutativity constraint is given by\label{exGcrossedproofc}
\begin{align}
vw = (g.w)(g.v) \myforall v\in \mathfrak{A}_g, \quad w \in \mathfrak{A}.
\end{align}\end{myenumerate}
The proof of the above statements follows directly from the definition of $\RepGrpd{K}$ if we recall that the multiplication is obtained by evaluation on the pair of pants, the pairing by evaluation on the bent cylinder etc.

We could ask whether the functor \eqref{exGcrossedproof1} from two-dimensional $G$-equivariant topological field theories to commutative Frobenius objects in $\RepGrpd{K}$ is an equivalence. Rephrasing Theorem~3.1 in \cite{turaevhqft} in the language of this section, it will be an equivalence once we restrict in range to those Frobenius objects satisfying the following additional properties:
\begin{myenumerate}
\item[{\normalfont (F1)}] Self-invariance of twisted sectors: Any element $g\in G$ acts trivially on $\mathfrak{A}_g$, i.e.\ $g.v=v$ for $v\in \mathfrak{A}_g$. This entails that the commutativity constraint takes the form 
\begin{align} vw = (g.w)v \myforall v\in \mathfrak{A}_g, \quad w \in \mathfrak{A}.\end{align}

\item[{\normalfont (F2)}] Trace property: For $g,h\in G$ and $v\in \mathfrak{A}_{ghg^{-1}h^{-1}}$ the equality\enlargethispage*{1cm}
\begin{align} 
\tr_{\mathfrak{A}_g} vh = \tr_{\mathfrak{A}_h} g^{-1}v
\end{align} holds, where $v$ is the multiplication map by $v$ from the left and $g$ and $h$ denote the automorphisms coming from the action with these elements.

\end{myenumerate}
If a Frobenius object in $\RepGrpd{K}$ satisfies these axioms, it is called \emph{crossed Frobenius $G$-algebra}. In summary, Theorem~3.1 in \cite{turaevhqft} implies that \eqref{exGcrossedproof1} becomes an equivalence once restricted in range to crossed Frobenius $G$-algebras. Hence, the change to equivariant coefficients allow us to play back large parts of the equivariant classification problem to the non-equivariant one. This leads to significant simplifications of the proof given in \cite[pages~40-64]{turaevhqft}. 
\end{example}

\subsection{Parallel section functor\label{secparsec}}
The definition of the pushforward operation relies crucially on a symmetric monoidal functor from the category $\RepGrpd{K}$ to $\Vect_K$ that we can obtain from the results in \cite{trova} which are based on ideas in \cite{fhlt}. To apply these results we need to assume that $K$ is a field of characteristic zero.
In fact, \cite{trova} provides two such functors which, in the setup of the present article, will coincide. Hence, we obtain one functor
\begin{align}
\Par : \RepGrpd{K} \to \Vect_K
\end{align} that we will refer to as \emph{parallel section functor}. We will express this functor in terms of pushforward maps and use these to provide concrete formulae for the pushforward operation and orbifold construction. The definition of the pushforward map make use of integrals with respect to groupoid cardinality which are recalled in Appendix~\ref{secgrpdkardint}.

\begin{definition}[Pushforward map]\label{defpushmap}
Let $\Phi : \Gamma \to \Omega$ be a functor between essentially finite groupoids and $K$ a field of characteristic zero. Then for any $K$-vector bundle $\varrho$ over $\Omega$ we define the \emph{pushforward map}
\begin{align}
\Phi_* :   \Par \Phi^* \varrho\to  \Par \varrho \end{align} by \begin{align}
(\Phi_* s)(y) = \sum_{[x,g]\in \pi_0(\Phi^{-1}[y])} \frac{\varrho(g)s(x)}{|\! \Aut(x,g)|} = \int_{\Phi^{-1}[y]} \varrho(g)s(x)\,\td (x,g)  \in \varrho(y)
\end{align} for any parallel section $s$ of $\Phi^* \varrho$ and $y\in \Omega$.\end{definition}

\begin{remarks}\label{bmkmscpushmap}

\item Obviously, $\Aut(x,g)\cong \Aut_0(x) := \ker (\Aut(x) \to \Aut(\Phi(x)))$ and hence
\begin{align} (\Phi_* s)(y)  = \sum_{[x,g]\in \pi_0(\Phi^{-1}[y])} \frac{\varrho(g)s(x)}{|\! \Aut_0(x)|}.\end{align} 

\item If $\Omega = \Gamma$ and $\Phi = \id_\Gamma$, then $\Phi^{-1}[y]$ is connected and $|\! \Aut_0(x)|=1$. This implies ${\id_\Gamma}_* = \id_{\Par \varrho}$.\label{bmkmscpushmap2}

\end{remarks}

\noindent For pushforward maps we can derive another formula that will be useful in the sequel. It can be verified by a direct computation: 

\begin{corollary}[]\label{kormscpushmap}
Let $\Phi : \Gamma \to \Omega$ be a functor between essentially finite groupoids and $K$ a field of characteristic zero. Then for any $K$-vector bundle $\varrho$ over $\Omega$ and any parallel section $s$ of $\Phi^* \varrho$ the formula
\begin{align}
(\Phi_* s)(y) = \sum_{\substack{[x]\in \pi_0(\Gamma) \\ \Phi(x)\stackrel{h}{\cong} y}} \sum_{g\in \Aut(y)} \frac{\varrho(g)\varrho(h) s(x)}{|\! \Aut(x)|} \myforall y\in \Omega \end{align} holds. Here in the first sum $\Phi(x)\stackrel{h}{\cong} y$ expresses the condition that such an isomorphism exists. Its choice is irrelevant. Note that we do not sum over all $h$.
\end{corollary}

\noindent For pullback maps we derived a composition law and a naturality condition in Proposition~\ref{pullbackvbgrpdmsc}. Both results have analogues for pushforward maps and can be checked by a direct computation:

\begin{proposition}[]\label{satzmsceigpushmaps}
Let $\Phi : \Gamma \to \Omega$ be a functor between essentially finite groupoids and $K$ a field of characteristic zero.
\begin{myenumerate}
\item The pushforward maps obey the composition law $(\Psi\circ \Phi)_* = \Psi_* \circ \Phi_*$, where $\Psi : \Omega \to \Lambda$ is another functor between essentially finite groupoids.\label{satzmsceigpushmapsa}

\item The pushforward maps are natural in the sense that they provide a natural transformation\label{satzmsceigpushmapsb}
\begin{align}
\Par_\Gamma \circ\  \Phi^* \to \Par_\Omega, 
\end{align} i.e.\ for any morphism $\lambda : \varrho \to \xi$ of vector bundles over $\Omega$ the square 
\begin{center}
\begin{tikzpicture}[scale=2]
\node (A1) at (0,1) {$\Par \Phi^*\varrho$};
\node (A2) at (1,1) {$\Par \varrho$};
%\node (A5) at (4,1) {$E$};
\node (B1) at (0,0) {$\Par \Phi^* \xi$};
\node (B2) at (1,0) {$\Par \xi$};
\path[->,font=\scriptsize]
(A1) edge node[above]{$\Phi_*$} (A2)
(A1) edge node[left]{$(\Phi^* \lambda)_*$} (B1)
%(A4) edge node[above]{$i$} (A5)
(B1) edge node[above]{$\Phi_*$} (B2)
(A2) edge node[right]{$\lambda_*$} (B2);
\end{tikzpicture}
\end{center}
commutes.
\end{myenumerate}
\end{proposition}

\noindent Using the push maps we can rephrase \cite[Theorem~5.4]{trova} for the symmetric monoidal category of finite-dimensional vector spaces over a field of characteristic zero: 

\begin{theorem}[Parallel section functor]\label{thmmscparsymfunctor}
For a field $K$ of characteristic zero the assignment
\begin{align}
\Par \ :\ \RepGrpd{K} &\to \Vect_K \\ (\Gamma,\varrho) & \mapsto \Par \varrho \\ \left((\Gamma_0,\varrho_0) \stackrel{r_0}{\longleftarrow} (\Lambda,\lambda) \stackrel{r_1}{\to} (\Gamma_1,\varrho_1)\right) & \mapsto ({r_1}_* \lambda_* r_0^* : \Par \varrho_0 \to \Par \varrho_1)
\end{align} yields a symmetric monoidal functor. We will refer to this functor as parallel section functor.   
\end{theorem}

\begin{remarks}\label{bmkmscofkexplizit}
\item For concrete computations in subsequent sections note we write out the above definition. The parallel section functor yields for a parallel section $s$ of $\varrho_0$ and $x_1 \in \Gamma_1$
\begin{align} (\Par (\Lambda,\lambda)s)(x_1) = \int_{r_1^{-1}[x_1]} \varrho_1(g) \lambda_y s(r_0(y))\,\td (y,g).\end{align} Using Corollary~\ref{kormscpushmap} we can also write
\begin{align} (\Par (\Lambda,\lambda)s)(x_1) = \sum_{\substack{[y]\in \pi_0(\Lambda) \\ r_1(y) \stackrel{h}{\cong} x_1}} \sum_{g\in \Aut(x_1)} \frac{\varrho(g)\varrho(h) s(r_0(y))}{|\! \Aut(y)|},\end{align} where the choice of $h$ is irrelevant and we do not sum over $h$. \label{bmkmscofkexplizit1}

\item In the following situation a strictification of \ref{bmkmscofkexplizit1} is possible: Let us assume that $r_1: \Lambda \to \Gamma_1$ is an isofibration (Remark~\ref{bmkmscrestrictionofcoverings}),
then for $y\in \Lambda$ with $r_1(y)\cong x_1$ we can find $y' \in \Lambda$ with $r_1(y')=x_1$, i.e.\ we can turn equalities holding up to isomorphism into strict ones. Using the identity as morphism $r_1(y') \to x_1$ we find in this case the strict version of the formula from \ref{bmkmscofkexplizit1}
\begin{align} (\Par (\Lambda,\lambda)s)(x_1) = \sum_{\substack{[y]\in \pi_0(\Lambda) \\ r_1(y) = x_1}} \sum_{g\in \Aut(x_1)} \frac{\varrho(g)s(r_0(y))}{|\! \Aut(y)|},\end{align} where $r_1(y)=x_1$ means that the representatives are chosen such that this equality holds, but the result does not depend on the representatives. \label{bmkmscofkexplizit2}

\item The isofibration property is always fulfilled in the cases we are interested in because for $G$-equivariant topological field theories all relevant functors between groupoids will be restriction functors
\begin{align}
\iota^* : \Pi(M,BG) \to \Pi(\Sigma,BG)
\end{align} for the inclusion $\iota : \Sigma \to M$ of a collection of boundary components of some oriented compact manifold $M$ with boundary. Since $\iota$ is a cofibration, we can extend homotopies of maps $\Sigma \to BG$ to homotopies of maps $M \to BG$. This implies that $\iota^*$ automatically has the needed lifting property.\label{bmkmscofkexplizit3} 

\end{remarks}

\section{The pushforward and the orbifold construction\label{secorbifoldconstruction}}
We have now established all the ingredients of the pushforward operation. For the rest of the article we will denote by $K$ a field of characteristic zero.

\subsection{Definition and properties}
We will define the pushforward of a $G$-equivariant topological field theory along an arbitrary morphism $\lambda : G \to H$ of finite groups. 

\begin{definition}[Pushforward for equivariant topological field theories]\label{defpushforward}
Let \begin{align} Z: G\text{-}\Cob(n) \to \Vect_K\end{align} be a $G$-equivariant topological field theory and $\lambda : G \to H$ a morphism of finite groups. Then the \emph{pushforward $\lambda_* Z$ of $Z$ along $\lambda$} is defined as the concatenation of symmetric monoidal functors
\begin{align}
\lambda_* : H\text{-}\Cob(n) \stackrel{\widehat{Z}^\lambda}{\to} \RepGrpd{K} \xrightarrow{\Par} \Vect_K,
\end{align} where $\widehat{Z}^\lambda : H\text{-}\Cob(n) \to \RepGrpd{K}$ was defined in Section~\ref{secchangetoequivcoeff} and $\Par$ is the parallel section functor (Theorem~\ref{thmmscparsymfunctor}).  \end{definition}

Unpacking the definition of the change to equivariant coefficients and the parallel section functor (use Remark~\ref{bmkmscofkexplizit}) yields the following explicit description of the pushforward operation:

\begin{proposition}[]\label{satzformaluepushoptft}
	For a morphism $\lambda : G \to H$ and a $G$-equivariant topological field theory $Z : G\text{-}\Cob(n) \to \Vect_K$ the pushforward theory $\lambda_* Z : H\text{-}\Cob(n) \to \Vect_K$ admits the following description:
	\begin{myenumerate} \item On an object $(\Sigma,\varphi)$ in $H\text{-}\Cob(n)$ the pushforward theory $\lambda_* Z$ is given by	\begin{align}
		(	\lambda_* Z)(\Sigma,\varphi) = \Par (\Phi^{-1}[\varphi] \to \Pi(\Sigma,BG) \xrightarrow{\varrho_\Sigma}  \Vect_K).
		\end{align} \label{satzformaluepushoptfta}
		\item On a morphism $(M,\psi) : (\Sigma_0,\varphi_0) \to (\Sigma_1,\varphi_1)$ in $H\text{-}\Cob(n)$ the pushforward theory $\lambda_* Z$ is given by	\begin{align}
		(\lambda_* Z)(M,\psi)s(\alpha_1,a_1) &= \int_{r_1^{-1}[\alpha_1,a_1]} Z(\Sigma_1 \times [0,1],g) Z(M,\beta) s(\beta|_{\Sigma_0},b|_{\Sigma_0}) \,\td (\beta,b,g)\\ \myforall &s \in (\lambda_* Z)(\Sigma_0,\varphi_0), \quad (\alpha_1,a_1) \in \lambda_*^{-1}[\varphi_1] .
		\end{align} In the particular case $\Sigma_0 = \Sigma_1 = \Sigma$ and $M=\Sigma \times [0,1]$, i.e.\ if $\psi$ is a homotopy $\varphi_0 \stackrel{h}{\simeq} \varphi_1$, this reduces to
		\begin{align} (\lambda_* Z)(M,h)s(\alpha_1,a_1) = s(\alpha_1,h^- * a_1), \end{align} where $*$ denotes the composition of homotopies.\label{satzformaluepushoptftb} 
		\item The pushforward theory $\lambda_* Z$ assigns to an $n$-dimensional closed oriented manifold $M$ and a map $\psi : M \to BH$ the invariant
		\begin{align} (\lambda_* Z)(M,\psi) = \int_{\lambda_*^{-1}[\psi]} Z(M,\phi) \,\td (\phi,h) \end{align}\label{satzformaluepushoptftc} which can be expressed as an integral with respect to groupoid cardinality.  \end{myenumerate}
\end{proposition}

\noindent We obtain the orbifold construction as the following special case of Definition~\ref{defpushforward}: 

\begin{definition}[Orbifold theory of an equivariant topological field theory]
	For a finite group $G$ the \emph{orbifold theory} $Z/G$ of a $G$-equivariant topological field theory $Z$ is defined to be the pushforward of $Z$ to the trivial group.
	\end{definition}

This finishes the construction of an orbifold theory $Z/G: \Cob(n) \to \Vect_K$ for a $G$-equivariant theory $Z: G\text{-}\Cob(n) \to \Vect_K$. We now give a very explicit description of the orbifold theory and derive its most important properties. To this end, we combine Proposition~\ref{satzformaluepushoptft} with the Remarks~\ref{bmkmscofkexplizit}:

\spaceplease

\begin{corollary}[]\label{korofkexplicit}
For a $G$-equivariant topological field theory $Z: G\text{-}\Cob(n) \to \Vect_K$ the orbifold theory
$Z/G : \Cob(n) \to \Vect_K$ admits the following description:
\begin{myenumerate} \item On an object $\Sigma$ in $\Cob(n)$ the orbifold theory $Z/G$ is given by the vector space $\Par \varrho_\Sigma$ of parallel sections of the vector bundle $\varrho_\Sigma$ over $\Pi(\Sigma,BG)$ given by $Z$ and Proposition~\ref{satzmscreprpigrpd}, i.e.\
$Z/G(\Sigma) = \Par \varrho_\Sigma$.
Explicitly,
\begin{align}
\frac{Z}{G}(\Sigma) \cong \bigoplus_{[\varphi]\in [\Sigma,BG]} Z(\Sigma,\varphi)^{\Aut(\varphi)} \cong \bigoplus_{[\varphi]\in [\Sigma,BG]} Z(\Sigma,\varphi)_{\Aut(\varphi)}.
\end{align} \label{korofkexplicita}

\item On a morphism $M : \Sigma_0 \to \Sigma_1$ in $\Cob(n)$ the orbifold theory $Z/G$ is given by	
\begin{align}\left(\frac{Z}{G}(M)s\right)(\varphi_1) &= \int_{r_1^{-1}[\varphi_1]} Z((\Sigma_1 \times [0,1]) \circ M, h\cup \psi) s(\psi|_{\Sigma_0}) \,\td (\psi,h) \\ %\myforall s\in \frac{Z}{G}(\Sigma_0)&=\Par \varrho_{\Sigma_0}, \quad \varphi_1 : \Sigma_1 \to BG
\end{align} holds. 
This expresses $Z/G(M)$ as an integral with respect to groupoid cardinality.
Here $r_1 : \Pi(M,BG) \to \Pi(\Sigma_1,BG)$ is the restriction functor, $h\cup \psi : (\Sigma_1 \times [0,1]) \circ M \to BG$ is the function on $M$ with a cylinder glued to it obtained from $\psi : M \to BG$ and a homotopy $\psi|_{\Sigma_1} \stackrel{h}{\simeq} \varphi_1$. As an alternative we can use the formula
\begin{align}
\left(\frac{Z}{G}(M)s\right)(\varphi_1) = \sum_{\substack{[\psi]\in[M,BG] \\ \psi|_{\Sigma_1} = \varphi_1}} \sum_{\varphi_1 \stackrel{h}{\simeq} \varphi_1 \in \Aut(\varphi_1)} \frac{ Z((\Sigma_1 \times [0,1]) \circ M, h\cup \psi ) s(\psi|_{\Sigma_0})}{|\! \Aut(\psi)|},
\end{align} where for the first sum the representatives are chosen such that $\psi|_{\Sigma_1} = \varphi_1$ holds strictly, but the result is independent of the representatives. \label{korofkexplicitb}

\item The orbifold theory $Z/G$ assigns to an $n$-dimensional closed oriented manifold $M$ the invariant
\begin{align} 
\frac{Z}{G}(M) = \int_{\Pi(M,BG)} Z(M,\psi)\,\td \psi
\end{align} \label{satzformaluepushoptftc} which can be expressed as an integral with respect to groupoid cardinality.  \label{korofkexplicitc}\end{myenumerate}
\end{corollary}

\spaceplease

\noindent The formulae are compatible with the idea of integrating over twisted sectors. \\[2ex]\begin{remark}[Twisted version of the orbifold construction]\label{extwisteddworbmsc}
Let $G$ be a finite group and $BG_\theta$ be the $n$-dimensional primitive $G$-equivariant theory associated to a cocycle $\theta \in Z^n(BG;\operatorname{U}(1))$ as defined in \cite[I.2.1]{turaevhqft}. 
Working over the field of complex numbers, the functor
\begin{align} \HSym(G\text{-}\Cob(n),\Vect_\mathbb{C}) \xrightarrow{?\otimes BG_\theta} \HSym(G\text{-}\Cob(n),\Vect_\mathbb{C})  \xrightarrow{?/G} \Sym(\Cob(n),\Vect_\mathbb{C}) \end{align} first takes the tensor product of a given equivariant topological field theory with the primitive theory associated to $\theta$ and then orbifoldizes. We call this functor the \emph{$\theta$-twisted orbifold construction}. 

Let us make the following observations:
\begin{xenumerate}
\item If we apply the $\theta$-twisted orbifold construction to the trivial $G$-equivariant theory, which assigns identities between complex lines to all morphisms with maps into $BG$, then we obtain the orbifold theory $BG_\theta / G$ of $BG_\theta$. The ordinary topological field theory $BG_\theta / G$ is commonly referred to as \emph{$\theta$-twisted Dijkgraaf-Witten theory} (see e.g. \cite{freedquinn,morton1}). It assigns to a closed oriented $n$-dimensional manifold $M$ the number
\begin{align}
\frac{BG_\theta }{ G} (M) = \int_{\Pi(M,BG)} \spr{\psi^* \theta,\mu_M}\,\td \psi,
\end{align}  where $\mu_M \in H_n(M)$ is the fundamental class of $M$. For an object $\Sigma$ in $\Cob(n)$ we obtain the vector space
\begin{align} \frac{BG_\theta}{G}(\Sigma) \cong \bigoplus_{[\varphi] \in [\Sigma,BG]} BG_\theta(\Sigma,\varphi)^{\Aut(\varphi)}.\end{align} \label{extwisteddworbmsc1}
\item If we denote by $U:G\text{-}\Cob(n) \to \Cob(n)$ the forgetful functor, then \begin{align} \frac{?}{G} \circ U^* \in \End(\Sym(\Cob(n),\Vect_\mathbb{C}))\end{align} assigns to the trivial $n$-dimensional topological field theory the Dijkgraaf-Witten theory for the group $G$ as follows from \ref{extwisteddworbmsc1} in the case $\theta =1$. For general $G$ this implies that $(?/G) \circ U^*$ is not naturally isomorphic to the identity. So $?/G$ cannot be an adjoint to the pullback along $U$ because a pair of adjoint functors between groupoids always form an equivalence. 
\end{xenumerate}
\end{remark}

\begin{example}[One-dimensional case]
	A one-dimensional $G$-equivariant topological field theory is equivalently a finite-dimensional $G$-representation by \cite[I.1.4]{turaevhqft}. Orbifoldization then simply amounts to taking (co)invariants of this representation. 
\end{example}

\begin{example}[Equivariant Dijkgraaf-Witten models]\label{orbifolddwmodels}
	In \cite{maiernikolausschweigerteq} for any short exact sequence  \begin{align} 0 \to G \to H \stackrel{\lambda}{\to} J \to 0\end{align} of finite groups a $J$-equivariant topological field theory $Z_\lambda$ with values in complex Hilbert spaces is constructed, the so-called \emph{$J$-equivariant Dijkgraaf-Witten theory}. The number it assigns to a top-dimensional closed oriented manifold $M$ and a map $\psi : M \to BJ$ is given by the groupoid cardinality $|\lambda_* ^{-1}[\psi]|$ of the homotopy fiber of the functor $\lambda_* : \Pi(M,BH) \to \Pi(M,BJ)$ over $\psi$. 
	By Corollary~\ref{korofkexplicit}, \ref{korofkexplicitc} the orbifold theory satisfies
	\begin{align}
	\frac{Z_\lambda}{ J}(M) = \int_{\Pi(M,BJ)}  |\lambda^{-1}_*[\psi]| \,\td \psi = |\Pi(M,BH)|,
	\end{align} where Cavalieri's principle (Proposition~\ref{satzmsccavalierisprinciple}) enters in the last step. 
	On the right hand side we see the invariant the 
	Dijkgraaf-Witten theory $Z_H$ for the group $H$ would assign to $M$. In fact, one can show by a tedious, but direct computation \cite{woike} that the orbifold theory of $Z_\lambda$ is given by $Z_H$. 
	\end{example}

\subsection{Dimension two: Orbifoldization of crossed Frobenius algebras\label{secmscofk2d}}
For two-dimensional equivariant topological field theories we can write down the orbifold construction on the level of classifying objects, i.e.\ by using the classification of $G$-equivariant topological field theories by crossed Frobenius $G$-algebras due to \cite{turaevhqft} as recalled in Example~\ref{exGcrossedproof}. 
The orbifold theory is an ordinary two-dimensional topological field theory and hence equivalent to a commutative Frobenius algebra. Our goal in this subsection is to determine this Frobenius algebra.  

\begin{definition}[Parallel sections of a crossed Frobenius $G$-algebra]\label{defmscparallelsectionfrob}
Let $G$ be a finite group and $\mathfrak{A}$ be a crossed Frobenius $G$-algebra.
A \emph{parallel section} of $\mathfrak{A}$ is a parallel section of the underlying vector bundle over $G//G$, i.e.\ a family $s=(s(g))_{g\in G}$ of vectors $s(g) \in \mathfrak{A}_g$ with $s(hgh^{-1})=h.s(g)$ for all $g,h\in G$. We denote the \emph{vector space of parallel sections} of $\mathfrak{A}$ by $\mathfrak{A}/G$. 
\end{definition}

\spaceplease
\begin{theorem}[Orbifold construction for two-dimensional equivariant topological field theories]\label{thmorbifoldconstructionfor2dtft}
Let $G$ be a finite group, $Z:G\text{-}\Cob(2) \to \Vect_K$ a two-dimensional $G$-equivariant topological field theory and $\mathfrak{A}$ its crossed Frobenius $G$-algebra. Then the orbifold theory $Z/G: G\text{-}\Cob(2) \to \Vect_K$ is classified by the commutative Frobenius algebra which, as a vector space, is given by the vector space of parallel section $\mathfrak{A}/G$ of $\mathfrak{A}$. The multiplication is given by
\begin{align}
(ss')(g) = \sum_{\substack{a,b\in G \\ ab=g}} s(a)s'(b) \myforall s,s'\in \frac{\mathfrak{A}}{G}, \quad g\in G,
\end{align} the unit is the parallel section with $s(1)=1$ and $s(g) =0$ for $g\neq 1$, and the pairing is given by 
\begin{align}
\kappa(s,s') = \frac{1}{|G|} \sum_{g\in G} \kappa(s(g),s(g^{-1})) \myforall s,s'\in \frac{\mathfrak{A}}{G}.
\end{align}
The Frobenius algebra $\mathfrak{A}/G$ is called the orbifold Frobenius algebra of $\mathfrak{A}$.
\end{theorem}

\begin{remarks}

\item To reduce notational complexity we refrain from introducing additional symbols for the multiplication, pairing etc. of the orbifold algebra. 

\item The term orbifold algebra is not only justified by the above assertion, but also used in the literature: By \cite[Proposition~2.1.3]{kaufmannorb} the invariants of a crossed Frobenius $G$-algebra naturally form a Frobenius algebra. This is exactly the idea underlying Definition~\ref{defmscparallelsectionfrob} because the holonomy principle allows us to identify invariants with parallel sections. In fact, on the level of vector spaces we have the non-canonical isomorphism
\begin{align}
\frac{\mathfrak{A}}{G} \cong \bigoplus_{[g] \in G/G} \mathfrak{A}_g^{\Aut(g)}.
\end{align} In summary, the orbifold construction is designed in such a way that it relates in dimension two to known orbifold constructions for Frobenius algebras. 

\end{remarks}

\begin{proof}
We know that the orbifold theory is a two-dimensional topological field theory. By the classification result for two-dimensional topological field theories it can be equivalently described by the commutative Frobenius algebra obtained by evaluation on the circle. Since the orbifold theory on objects is given by forming spaces of parallel sections, we deduce that this commutative Frobenius algebra, as a vector space, is $\mathfrak{A}/G$. The multiplication is obtained by evaluation on the pair of pants. In Example~\ref{exGcrossedproof} we have seen that the application of the stack of $G$-bundles to this bordism yields the span
\begin{align} G//G \times G//G \stackrel{B}{\longleftarrow} (G\times G)//G \stackrel{M}{\to} G//G, \end{align} where $B$ is the obvious functor and $M$ the multiplication. Using the explicit formula for the orbifold construction (Corollary~\ref{korofkexplicit}, \ref{korofkexplicitb}) we find
\begin{align}
(ss')(g) = \sum_{\substack{[a,b] \in (G\times G)//G \\ ab=g \\ h \in \Aut(g)}} \frac{h.(s(a)s'(b))}{|\! \Aut(a,b)|} \myforall s,s'\in \frac{\mathfrak{A}}{G}, \quad g\in G,
\end{align} where the representatives are chosen such that $ab=g$ holds strictly. Now denote by $\Gamma_g$ the full subgroupoid of $(G\times G)//G$ of all $(a,b)$ with $ab=g$ (in fact, this is equivalent to the full subgroupoid of all $(a,b)$ with $ab=cgc^{-1}$ for some $c\in G$), define
\begin{align}
f(s,s',g)(a,b) := \sum_{h\in \Aut(g)} h.(s(a)s'(b)) \myforall (a,b) \in \Gamma_g
\end{align} and observe that $f(s,s',g): \Gamma_g \to \mathfrak{A}_g$ is an invariant function on $\Gamma_g$. Having introduced this notation we obtain
\begin{align} (ss')(g) = \int_{\Gamma_g} f(s,s',g)= \sum_{[a,b] \in \Gamma_g} \frac{1}{|\! \Aut(a,b)|} \sum_{h\in \Aut(g)} s (hah^{-1}) s'(hbh^{-1}). \label{eqnmscofkdim22}\end{align} The groupoid $\Gamma_g$ is the action groupoid of the action of $\Aut(g)$ on the underlying object set by conjugation. Recalling the classical orbit theorem stating that the map
\begin{align}
\Aut(g) \ni h \mapsto (hah^{-1},hbh^{-1}) \in \mathcal{O}(a,b)
\end{align} induces a bijection $\Aut(g) / \Aut(a,b) \cong \mathcal{O}(a,b)$ we see that we can replace the inner sum by
\begin{align}
\sum_{h\in \Aut(g)} s (hah^{-1}) s'(hbh^{-1}) = |\! \Aut(a,b)| \sum_{(x,y) \in \mathcal{O}(a,b)} s(x)s'(y).
\end{align} Using this together with \eqref{eqnmscofkdim22} proves the formula
\begin{align}
(ss')(g) = \sum_{[a,b]\in \Gamma_g} \sum_{(x,y) \in \mathcal{O}(a,b)} s(x)s'(y) = \sum_{\substack{a,b\in G \\ ab=g}} s(a)s'(b)
\end{align} for the multiplication. The formula for the pairing can be derived in a similar way.   
\end{proof}

\subsection{The composition law for the pushforward operation}
For a morphism $\lambda : G \to H$ of finite groups we have introduced in Definition~\ref{defpushforward} the {pushforward of $G$-equivariant topological field theories along $\lambda$} as a functor
 \begin{align}
 \lambda_* : \HSym(G\text{-}\Cob(n),\Vect_K) \to \HSym(H\text{-}\Cob(n), \Vect_K) 
 \end{align}
where the orbifold construction corresponds to the case $H=1$. 
For morphisms $\lambda : G \to H$ and $\mu : H \to J$ of finite groups it is reasonable to ask whether the pushforward operations on equivariant topological field theories obey
\begin{align}
	(\mu \circ \lambda)_*  \cong \mu_* \circ \lambda_*.
\end{align} Since the push operations are not left-adjoint to the pullback functors, such a relation does not hold automatically. However, we will show below in Theorem~\ref{thmgluinglawpush} that this composition law holds. Let us first investigate the situation at the level of manifold invariants. For this we need the following assertion which is a consequence of the pasting law for (homotopy) pullback squares:

\begin{lemma}[]\label{lemmahtppullbackcomposition}
	For functors $\Phi : \Gamma \to \Omega$ and $\Psi : \Omega \to \Lambda$ between groupoids, there is a canonical equivalence
	\begin{align}
	(\Psi \circ \Phi)^{-1}[z] \cong \Psi^{-1}[z] \times_\Omega \Gamma.
	\end{align}
\end{lemma}

\noindent If $Z : G\text{-}\Cob(n) \to \Vect_K$ is a $G$-equivariant field theory and $M$ a closed $n$-dimensional oriented manifold together with a map $\psi : M \to BJ$, we obtain
\begin{align}
	( (\mu \circ \lambda)_*  Z)(M,\psi) = \int_{(\mu \circ \lambda)_*^{-1} [\psi]} Z(M,\phi)\,\td (\phi,g)
	\end{align} by Proposition~\ref{satzformaluepushoptft}, \ref{satzformaluepushoptftc}. Using Lemma~\ref{lemmahtppullbackcomposition} and the transformation formula (Proposition~\ref{satzmsctransformationformula}) we find
\begin{align}
	( (\mu \circ \lambda)_*  Z)(M,\psi) =  \int_{\mu_*^{-1} [\psi] \times _{\Pi(M,BH)} \Pi(M,BG)} Z(M,\alpha)\,\td (\beta,h,\alpha,g).
	\end{align} Next we apply the generalized Cavalieri's principle (Proposition~\ref{satzmscgencavalierisprinciple}) to the projection functor 
	\begin{align}P : \mu_*^{-1} [\psi] \times _{\Pi(M,BH)} \Pi(M,BG) \to \mu^{-1}_* [\psi].\end{align} Since by the universal property of the homotopy pullback the fiber $P^{-1}[\beta,h]$ of $P$ over some $(\beta,h) \in \mu_*^{-1} [\psi]$ is just the fiber $\lambda_*^{-1} [\beta]$ of $\lambda_*$ over $\beta$ we find the desired relation for invariants
\begin{align}
		( (\mu \circ \lambda)_*  Z)(M,\psi) &= \int_{\mu_*^{-1} [\psi] } \int_{\lambda_*^{-1} [\beta]} Z(M,\alpha )\td (\alpha,g) \, \td (\beta,h) \\&= \int_{\mu_*^{-1} [\psi] } (\lambda_* Z)(M,\beta) \, \td (\beta,h) \\&= ((\mu_* \circ \lambda_*) Z)(M,\psi).
	\end{align}
A tedious, but direct computation based on Proposition~\ref{satzformaluepushoptft} shows that this result extends beyond invariants:

\begin{theorem}[]\label{thmgluinglawpush}
	For composable morphisms $G \stackrel{\lambda}{\to} H \stackrel{\mu}{\to} J$ of finite groups the pushforward operations on equivariant topological field theories obey
	\begin{align}
	(\mu \circ \lambda)_*  \cong \mu_* \circ \lambda_* .
	\end{align}
	\end{theorem}

	\begin{remark}
	The isomorphisms $(\mu \circ \lambda)_*  \cong \mu_* \circ \lambda_*$ appearing in Theorem~\ref{thmgluinglawpush} fulfill coherence conditions, i.e.\ they are part of the data of a 2-functor
	\begin{align}
	\catf{FinGrp} \to \Grpd 
	\end{align} from the the category of finite groups (seen as a bicategory with only trivial 2-morphisms) to the bicategory of groupoids, functors and natural isomorphisms. This 2-functor sends a finite group $G$ to the groupoid of $G$-equivariant topological field theories and a group morphism $\lambda : G \to H$ to the push functor $\lambda_*$. 
	\end{remark}
	
	From the sequence $1 \to G \to 1$ of group morphisms we immediately obtain: 
	
	\begin{corollary}[]\label{koresssurj}
	Let $G$ be a finite group. Then the orbifold construction \begin{align}\frac{?}{G} : \HSym(G\text{-}\Cob(n),\Vect_K) \to \Sym(\Cob(n),\Vect_K)\end{align} is essentially surjective, i.e.\ any topological field theory can be seen as an orbifold theory of an equivariant topological field theory for any given group. 
	\end{corollary}

\begin{remark}
	In \cite{stolzteichner} the idea is advocated to understand homotopy quantum field theories with target space $T$ (or certain equivalence classes thereof) as the value of a generalized cohomology theory on $T$. Indeed, a map $f: S \to T$ of target spaces give rise to an obvious pullback functor
	\begin{align}
	f^* : \HSym(T\text{-}\Cob(n),\Vect_K) \to \HSym(S\text{-}\Cob(n),\Vect_K) . 
	\end{align} It is, however, a lot more difficult to exhibit a pushforward map 
	\begin{align}
	f_* : \HSym(S\text{-}\Cob(n),\Vect_K) \to \HSym(T\text{-}\Cob(n),\Vect_K) \label{eqnpushmapsmap}
	\end{align} going in the opposite direction. The special case $T=\star$, i.e.\ the push to the point, should be understood as orbifoldization. 
	If $S$ and $T$ are aspherical spaces with finite fundamental group, then our pushforward provides a rigorous construction of all push maps \eqref{eqnpushmapsmap} because up to homotopy all maps $f: S \to T$ are induced by group morphisms. 
	\end{remark}

\begin{appendix}

\section{Groupoid cardinality and its integration theory\label{secgrpdkardint}}
The groupoid cardinality is a rational number that is assigned to an essentially finite groupoid. For more background on the groupoid cardinality we refer to \cite{baezgroup}.

For the characterization of groupoid cardinality we need the notion of a covering of groupoids.
Coverings of simplicial sets are defined in \cite[Appendix~I.2]{gabrielzisman}.
If we specialize to nerves of groupoids and take into account that these are always 2-coskeletal we arrive at the definition of covering of groupoids below. By $\Delta_n$ we will denote the standard simplex. Note that $\Delta_0=\star$ is the terminal object in the category of simplicial sets.
We also denote this object by 0 whenever we want to identify it with the zero vertex of $\Delta_1$ via the inclusion $0 \to \Delta_1$. For a groupoid $\Gamma$ we denote by $B\Gamma$ its nerve.

\begin{definition}[Covering of groupoids]
A functor $Q: \Gamma \to \Omega$ between (small) groupoids is called \emph{covering} if it is surjective on objects and if in any commuting square in $\sSet$ of the form
\begin{center}
\begin{tikzpicture}[scale=1.5]
\node (A1) at (0,1) {$0$};
\node (A2) at (1,1) {$B\Gamma$};
%\node (A5) at (4,1) {$E$};
\node (B1) at (0,0) {$\Delta_1$};
\node (B2) at (1,0) {$B\Omega$};
\path[->,font=\scriptsize, right hook->]
(A1) edge node[left]{$$} (B1);
\path[->,font=\scriptsize]
(A1) edge node[above]{$$} (A2)
%(A4) edge node[above]{$i$} (A5)
(B1) edge node[above]{$$} (B2)
(A2) edge node[right]{$BQ$} (B2);
\path[->,font=\scriptsize,dashed]
(B1) edge node[above]{$\exists!$} (A2);
\end{tikzpicture}
\end{center} the indicated lift exists and is unique (\emph{unique path lifting property}). The latter condition can be reformulated by requiring that for every morphism $g : y_0 \to y_1$ in $\Omega$ and any given $x_0$ with $Q(x_0)=y_0$, there is a unique morphism $g^* : x_0 \to x_1$ such that $Q(g^*)=g$. We say that a covering $Q: \Gamma \to \Omega$ is \emph{$n$-fold} or \emph{$n$-sheeted} if $Q^{-1}(y)$ contains $n$ objects for every $y \in \Omega$. 
\end{definition}

\spaceplease

\begin{remark}\label{bmkmscrestrictionofcoverings}
In the so-called \emph{canonical model structure} on the category of (small) groupoids, see \cite[14.1]{bousfieldcatmod}, the fibrations are the so-called \emph{isofibrations}, i.e.\ functors $Q:\Gamma \to \Omega$ between groupoids such that the lift in the diagram
\begin{center}
\begin{tikzpicture}[scale=1.5]
\node (A1) at (0,1) {$0$};
\node (A2) at (1,1) {$B\Gamma$};
%\node (A5) at (4,1) {$E$};
\node (B1) at (0,0) {$\Delta_1$};
\node (B2) at (1,0) {$B\Omega$};
\path[->,font=\scriptsize, right hook->]
(A1) edge node[left]{$$} (B1);
\path[->,font=\scriptsize]
(A1) edge node[above]{$$} (A2)
%(A4) edge node[above]{$i$} (A5)
(B1) edge node[above]{$$} (B2)
(A2) edge node[right]{$BQ$} (B2);
\path[->,font=\scriptsize,dashed]
(B1) edge node[above]{$$} (A2);
\end{tikzpicture}
\end{center} always exists. Uniqueness is not required. Hence, a covering is a special type of isofibration. 
\end{remark}

\noindent Homotopy fibers give us a source of coverings.

\begin{lemma}[]\label{lemmacoveringslicegroupoids}
Let $\Phi: \Gamma \to \Omega$ be a functor between small groupoids. Then for $y\in \Omega$ the forgetful functor
$\Phi^{-1}[y] \to \Gamma_y$ is an $|\!\Aut(y)|$-fold covering. 
\end{lemma}

\begin{definition}[Essentially finite groupoid]\label{defessentiallyfinitegroupoid}
A groupoid $\Gamma$ is called \emph{essentially finite} if $\pi_0(\Gamma)$ and $\Aut(x)$ for every $x\in \Gamma$ are finite. By $\FinGrpd$ we denote the \emph{category of essentially finite groupoids}. 
\end{definition}

\noindent Both the definition as well as the most important properties of groupoid cardinality can be summarized as follows: 

\begin{proposition}[Groupoid cardinality]
There is exactly one assignment $|?|$ of a rational number to each essentially finite groupoid satisfying the following conditions:
\begin{myenumerate}
\item[\normalfont (N)] For the groupoid $\star$ with one object and trivial automorphism group we have $|\star|=1$.
\item[\normalfont (E)] For equivalent essentially finite groupoids $\Gamma$ and $\Omega$ the equality $|\Gamma|=|\Omega|$ holds.
\item[\normalfont (U)] For the disjoint union $\Gamma \coprod \Omega$ of essentially finite groupoids we have $|\Gamma \coprod \Omega|=|\Gamma|+|\Omega|$.
\item[\normalfont (C)] For any $n$-fold covering $Q: \Gamma \to \Omega$ of essentially finite groupoids we have $|\Gamma|=n|\Omega|$. 
\end{myenumerate} This number is called groupoid cardinality. For an essentially finite groupoid $\Gamma$ it can be computed by
\begin{align} |\Gamma| = \sum_{[x] \in \pi_0(\Gamma)} \frac{1}{|\! \Aut(x)|}.\end{align} Moreover, for essentially finite groupoids $\Gamma$ and $\Omega$ the product $\Gamma \times \Omega$ is also essentially finite and its groupoid cardinality is given by $|\Gamma \times \Omega | = |\Gamma| |\Omega|$.
\end{proposition}

\noindent The groupoid cardinality gives rise to a rather primitive, but useful integration theory. 

\begin{definition}[Integral of invariant functions over groupoids with respect to groupoid cardinality]
An \emph{invariant function} $f$ on a groupoid $\Gamma$ with values in a vector space $V$ over a field of characteristic zero is the assignment of a vector $f(x)\in V$ to each $x\in \Gamma$ such that $f(x)=f(y)$ if $x\cong y$ in $\Gamma$. If $\Gamma$ is essentially finite, we define by
\begin{align} \int_\Gamma f = \int_\Gamma f(x) \,\td x := \sum_{[x]\in \pi_0(\Gamma)} \frac{f(x)}{|\! \Aut(x)|} \in V \end{align} the \emph{integral of $f$ over $\Gamma$}. 
\end{definition}

\begin{remarks}\label{bmkintegralofinvariantfunctionsmsc}

\item The integral is a linear functional on the vector space of invariant functions.\label{bmkintegralofinvariantfunctionsmsc1}

\item It is clear that an invariant function $f$ on a groupoid $\Omega$ can be pulled back along a functor $\Phi : \Gamma \to \Omega$ between groupoids. The result $\Phi^* f$ is an invariant function on $\Gamma$.

\end{remarks}

\noindent We need the following results that we call, in allusion to results known for instance from Lebesgue integration theory, transformation formula and Cavalieri's principle. The transformation formula explains how the integral behaves under pullback along equivalences. Its proof follows directly from the definitions. 

\begin{proposition}[Transformation formula]\label{satzmsctransformationformula}
Let $\Phi : \Gamma \to \Omega$ be an equivalence of essentially finite groupoids and $f: \Omega \to V$ an invariant function taking values in a vector space $V$ over a field of characteristic zero. Then the transformation formula
\begin{align} \int_\Gamma \Phi^* f = \int_\Omega f\end{align} holds. 
\end{proposition}

\noindent Cavalieri's principle from ordinary integration theory states that we can compute the volume of an object by summing (or rather integrating) the volume of its slices. The same is possible for groupoid cardinality instead of the volume.

\begin{proposition}[Cavalieri's principle]\label{satzmsccavalierisprinciple}
Let $\Phi : \Gamma \to \Omega$ be a functor of essentially finite groupoids. Then
\begin{align}
|\Gamma| = \int_\Omega |\Phi^{-1}[y]|\,\td y.
\end{align}
\end{proposition}

\noindent We omit the proof because the above assertion is a special case of:

\begin{proposition}[Generalized Cavalieri's principle]\label{satzmscgencavalierisprinciple}
Let $\Phi : \Gamma \to \Omega$ be a functor of essentially finite groupoids and $f: \Gamma  \to V$ an invariant function taking values in a vector space over characteristic zero. Then
\begin{align}
\int_\Gamma f = \int_\Omega \int_{\Phi^{-1}[y]} q_y^* f \,\td y,
\end{align} where $q_y : \Phi^{-1}[y] \to \Gamma$ is the canonical functor from the homotopy fiber over $y$ to $\Gamma$. 
\end{proposition}

\begin{proof}
Without loos of generality we may assume that $V$ is the ground field and $f= \delta_{[x]}$ for some $x\in \Gamma$. Then we can restrict $\Phi$ to the full subgroupoid $\Gamma_x$ of objects in $\Gamma$ isomorphic to $x$ and obtain
\begin{align}
\int_\Gamma f = |\Gamma_x| = \int_{\Omega} \left| \left(\Phi|_{\Gamma_x} \right) ^{-1} [y]\right|\,\td y
\end{align} by Cavalieri's principle. But $\left(\Phi|_{\Gamma_x} \right) ^{-1} [y]$ is the full subgroupoid of $\Phi^{-1}[y]$ of all objects projecting under $q_y : \Phi^{-1}[y] \to \Gamma$ to an object isomorphic to $x$. This implies $\left| \left(\Phi|_{\Gamma_x} \right) ^{-1} [y]\right| = \int_{\Phi^{-1}[y]} q_y^* f$ and proves the claim.   
\end{proof}

\noindent The integral of invariant functions with respect to groupoid cardinality is also compatible with coverings.

\begin{proposition}[]\label{lemmapromscadjointslice}
Let $Q: \Gamma \to \Omega$ be an $n$-fold covering of essentially finite groupoids. Then for any invariant function $f: \Omega \to V$ taking values in a vector space $V$ over a field $K$ of characteristic zero the integral formula
\begin{align}
\int_\Gamma Q^* f = n \int_{\Omega} f
\end{align}holds. 
\end{proposition} 

\begin{proof} Repeating the argument of the preceding proofs we can assume without loss of generality that $V=K$ and $f=\delta_{[y_0]}$ for some $y_0\in \Omega$. Let $\Omega_0$ be the full subgroupoid of $\Omega$ of all objects equal to $y_0$ (this is $y_0 // \Aut(y_0)$) and $Q_0 : \Gamma_0 \to \Omega_0$ the restriction of $Q$ (it is also an $n$-fold covering). Using the covering property of the groupoid cardinality we obtain
\begin{align} \int_\Gamma Q^* f = |\Gamma_0| = n|\Omega_0| = \frac{n}{|\! \Aut(y_0)|} = n \int_{\Omega} f.   \end{align}
\end{proof}

\end{appendix}

\extremespace

\end{document}